\documentclass[12pt]{amsart}
\usepackage{a4wide}
\usepackage{enumerate}
\usepackage{graphicx}
\usepackage{color}
\usepackage{amsmath}
\usepackage{comment}
\usepackage{soul}
\allowdisplaybreaks

\let\pa\partial  
\let\na\nabla  
\let\eps\varepsilon  
\newcommand{\N}{{\mathbb N}}  
\newcommand{\R}{{\mathbb R}} 
\newcommand{\diver}{\operatorname{div}}

\newcommand{\T}{{\mathbb T}}
\newcommand{\F}{{\mathcal F}}
\newcommand{\DD}{{\mathcal D}}

\newcommand{\dt}{\triangle t}

\newtheorem{theorem}{Theorem}   
\newtheorem{lemma}[theorem]{Lemma}   
\newtheorem{proposition}[theorem]{Proposition}   
\newtheorem{remark}[theorem]{Remark}

\begin{document}  

\title[Energy-transport systems for optical lattices]{Energy-transport systems 
for optical lattices: derivation, analysis, simulation}

\author{Marcel Braukhoff}
\address{Mathematisches Institut, Universit\"at zu K\"oln, Weyertal 86-90, 50931 K\"oln,
Germany}
\email{mbraukho@math.uni-koeln.de} 
\author{Ansgar J\"ungel}
\address{Institute for Analysis and Scientific Computing, Vienna University of  
	Technology, Wiedner Hauptstra\ss e 8--10, 1040 Wien, Austria}
\email{juengel@tuwien.ac.at} 

\date{\today}

\thanks{The authors acknowledge partial support from   
the Austrian Science Fund (FWF), grants P27352, P30000, and W1245.} 

\begin{abstract}
Energy-transport equations for the transport of fermions in optical lattices
are formally derived from a Boltzmann transport equation with a periodic
lattice potential in the diffusive
limit. The limit model possesses a formal gradient-flow structure like in the
case of the energy-transport equations for semiconductors. At the zeroth-order
high temperature limit, the energy-transport equations reduce to the whole-space
logarithmic diffusion equation which has some unphysical properties. Therefore,
the first-order expansion is derived and analyzed. The existence of weak
solutions to the time-discretized system for the particle and energy densities
with periodic boundary conditions
is proved. The difficulties are the nonstandard degeneracy
and the quadratic gradient term. The main tool of the proof is a result on
the strong convergence of the gradients of the approximate solutions.
Numerical simulations in one space dimension
show that the particle density converges to a constant steady state if
the initial energy density is sufficiently large, otherwise the particle
density converges to a nonconstant steady state. 
\end{abstract}

 

\maketitle


\section{Introduction}

An optical lattice is a spatially periodic structure that is formed by interfering
optical laser beams. The interference produces an optical standing wave that may trap
neutral atoms \cite{Blo05}. The lattice potential mimics the crystal lattice
in a solid, while the trapped atoms mimic the valance electrons in a solid state
crystal. In contrast to solid materials, it is easily
possible to adjust the geometry and depth of the potential of an optical lattice.
Another advantage is that the dynamics of the atoms, if cooled down to a few
nanokelvin, can be followed on the time scale of milliseconds.
Therefore, optical lattices are ideal systems to study physical phenomena that 
are difficult to observe in solid crystals. 
Moreover, they are promising candidates to realize
quantum information processors \cite{Jak04} and extremely precise atomic clocks
\cite{ADH15}.

The dynamics of ultracold fermionic clouds in an optical lattice can be modeled
by the Fermi-Hubbard model with a Hamiltonian that is a result of the
lattice potential created by interfering laser beams and short-ranged
collisions \cite{DGH15}. 
In the semi-classical picture, the interplay between diffusive and ballistic
regimes can be described by a Boltzmann transport equation \cite{GNZ09}, 
which is able to model qualitatively the observed cloud shapes \cite{SHR12}.

In this paper, we investigate
moment equations which are formally derived from a Boltzmann equation
in the diffusive regime. 
The motivation of our work is the observation that in the (relative)
high-temperature limit, the lowest-order diffusion
approximation of the Boltzmann equation leads to a logarithmic diffusion equation
\cite{MRR11}
which has nonphysical properties in the whole space
(for instance, it loses mass). Our aim is
to derive the next-order approximation, leading to energy-transport
equations for the particle and energy densities, and to analyze and 
simulate the resulting system of degenerate parabolic equations 
under periodic boundary conditions \cite{MRR11}.

The starting point is the (scaled) Boltzmann equation for the distribution function
$f(x,p,t)$,
\begin{equation}\label{1.be}
  \pa_t f + u\cdot\na_x f + \na_x V\cdot\na_p f = Q(f),
\end{equation}
where $x\in\R^d$ is the spatial variable, $p$ is the crystal momentum
defined on the $d$-dimensional torus $\T^d$ with unit measure, 
and $t>0$ is the time. The velocity $u$ is defined by $u(p)=\na_p\eps(p)$
with the energy $\eps(p)$,
$V(x,t)$ is the potential, and $Q(f)$ is the collision operator. 
Compared to the semiconductor Boltzmann
equation, there are two major differences.

First, the band energy $\eps(p)$ is given by the periodic dispersion relation
\begin{equation}\label{1.eps}
  \eps(p) = -2\eps_0\sum_{i=1}^d \cos(2\pi p_i), \quad p\in\T^d.
\end{equation}
The constant $\eps_0$ is a measure for the tunneling rate of a particle from one
lattice site to a neighboring one. In semiconductor physics, usually a
parabolic band structure is assumed, $\eps(p)=\frac12|p|^2$ \cite{Jue09}.
This formula also appears in kinetic gas theory as the (microscopic) kinetic
energy. The band energy in optical lattice is bounded, while it is unbounded
when $\eps(p)=\frac12|p|^2$. This has important consequences regarding the
integrability of the equilibrium distribution (see below).

Second, the potential $V$ is given by $V=-U_0n$, where $n=\int_{\T^d}fdp$ is the 
particle density and $U_0>0$ models the strength of the on-site interaction
between spin-up and spin-down components \cite{SHR12}. In semiconductor
physics, $V$ is the electric potential which is a given function
or determined self-consistently from the (scaled) Poisson equation \cite{Jue09}.
The definition $V=-U_0 n$ leads to unexpected ``degeneracies'' in the moment
equations, see the discussion below.

The collision operator is given as in \cite{SHR12} by the relaxation-time approximation
$$
  Q(f) = \frac{1}{\tau}(\F_f-f),
$$
where $\tau>0$ is the relaxation time
and $\F_f$ is determined by minimizing the free energy for fermions associated to
\eqref{1.be} under the constraints of mass and energy conservation (see Section 
\ref{sec.deriv} for details), leading to
$$
  \F_f(x,p,t) = \big(\eta + \exp(-\lambda_0(x,t)-\lambda_1(x,t)\eps(p))\big)^{-1},
  \quad x\in\R^d,\ p\in\T^d,\ t>0,
$$
where $(\lambda_0,\lambda_1)$ are the Lagrange multipliers resulting from the
mass and energy constraints. 
For $\eta=1$, we obtain the Fermi-Dirac distribution, while for
$\eta=0$, $\F_f$ equals the Maxwell-Boltzmann distribution.
We may consider $\F_f$ as a function of $(\lambda_0,\lambda_1)$ and write
$\F(\lambda_0,\lambda_1;p)=[\eta+\exp((-\lambda_0-\lambda_1\eps(p))]^{-1}$.

The variable $\lambda_1$ can be interpreted as the negative inverse 
(absolute) temperature,
while $\lambda_0$ is related to the so-called chemical potential \cite{Jue09}.
Since the band energy is bounded, the equilibrium $\F_f$ is integrable even when
$\lambda_1>0$, which means that the absolute temperature may be negative.
In fact, negative absolute temperatures can be realized in experiments with
cold atoms \cite{RMR10}. Negative temperatures occur in
equilibrated (quantum) systems that are characterized by an inverted population of
energy states. The thermodynamical implications of negative temperatures 
are discussed in \cite{Ram56}. 

In the following, we detail the main results of the paper.

\subsection*{Formal derivation and entropy structure}

Starting from the Boltzmann equation \eqref{1.be}, we derive 
formally moment equations in the limit
of large times and dominant collisions. More precisely, the particle density 
$n=\int_{\T^d}\F dp$ and energy density $E=\int_{\T^d}\F\eps dp$ solve
the energy-transport equations
\begin{equation}\label{1.et1}
  \pa_t n + \diver J_n = 0, \quad \pa_t E + \diver J_E - J_n\cdot\na V = 0
\end{equation}
for $x\in\R^d$, $t>0$, where the particle and energy current densities are
given by
\begin{equation}\label{1.et2}
  J_n = -\sum_{i=0}^1 D_{0i}\na\lambda_i - \lambda_1 D_{00}\na V,\quad
  J_E = -\sum_{i=0}^1 D_{1i}\na\lambda_i - \lambda_1 D_{10}\na V,
\end{equation}
and the diffusion coefficients depend nonlocally on $\F$ and hence on
$(\lambda_0,\lambda_1)$; see Proposition \ref{prop.deriv}. The structure of 
system \eqref{1.et1}-\eqref{1.et2} is similar to the semiconductor case \cite{JKP11} 
but the diffusion coefficients $D_{j}$ are different.
For $V=-U_0n$, the Joule heating term $J_n\cdot\na V$ contains the squared 
gradient $|\na n|^2$, while in the semiconductor case it contains $|\na V|^2$
which is generally smoother than $|\na n|^2$.

System \eqref{1.et1}-\eqref{1.et2} possesses a formal gradient-flow or entropy
structure. Indeed, the entropy $H$, defined in
Section \ref{sec.ent}, is nonincreasing in time,
$$
  \frac{dH}{dt} = -\int_{\R^d}\sum_{i,j=0}^1 \na\mu_i^\top L_{ij}\na\mu_j dx\le 0;
$$
see Proposition \ref{prop.ent}. Here, the functions 
$\mu_0=\lambda_0+\lambda_1 V$ and $\mu_1=\lambda_1$ are called the dual
entropy variables, and the coefficients $L_{ij}$ are defined in \eqref{2.L}.
In the dual entropy variables, the potential terms are eliminated, leading to
the ``symmetric'' problem
$$
  \pa_t\begin{pmatrix} n \\E \end{pmatrix} 
  = \diver\bigg(\begin{pmatrix} L_{00} & L_{01} \\ L_{10} & L_{11} \end{pmatrix}
  \na\begin{pmatrix} \mu_0 \\ \mu_1\end{pmatrix}\bigg),
$$
where the matrix $(L_{ij})$ is symmetric and positive definite.
This formal gradient-flow structure allows for the development
of an existence theory but only for uniformly positive definite diffusion matrices
\cite{DGJ97}. A general existence result (including electric potentials) 
is still missing. 

A further major difficulty comes from the fact that the
system possesses certain ``degeneracies'' in the mapping 
$(n,E)\mapsto \mu=(\mu_0,\mu_1)$ 
and the entropy production $-dH/dt$. For instance, the determinant of the
Jacobi matrix $\pa(n,E)/\pa\mu$ may vanish at certain points. 
Such a situation also occurs for the semiconductor energy-transport equations
but only at the {\em boundary} of the domain of definition (namely at $E=0$).
In the present situation, the degeneracy may occur at points in the {\em interior}
of the domain of definition. In view of these difficulties, an analysis of the general
energy-transport model \eqref{1.et1}-\eqref{1.et2} is currently out of reach.
This motivates our approach to introduce a simplified model.

\subsection*{Analysis of high-temperature energy-transport models}

We show the existence of weak solutions to a simplified
energy-transport model. It is argued in \cite{MRR11} that the temperature
is large (relative to the nanokelvin scale)
in the center of the atomic cloud for long times. Therefore, 
we simplify \eqref{1.et1}-\eqref{1.et2} by performing the high-temperature limit.
For high temperatures, the relaxation time may be approximated by
$\tau(n)=\tau_0/(n(1-n))$ \cite[Suppl.]{SHR12}. 
As $\theta=-1/\lambda_1$ can be interpreted as the 
temperature, the high-temperature limit corresponds to the limit $\lambda_1\to 0$.
Expanding $\F(\lambda_0,\lambda_1)$ around $(\lambda_0,0)$ up to zeroth order 
leads to the diffusion equation (see Section \ref{sec.high})
\begin{equation}\label{1.log}
  \pa_t n = \diver\bigg(\frac{\tau_0\na n}{n(1-\eta n)}\bigg) \quad\mbox{in }\R^d.
\end{equation}
In the case $\eta=0$, we obtain the logarithmic diffusion equation
$\pa_t n = \tau_0\Delta\log n$ which predicts a nonphysical behavior.
Indeed, in two space dimensions, it can be shown that
the particle number is not conserved
and the unique smooth solution exists for finite time only; see, e.g.,
\cite{DaDe99,Vaz06}. We expect a similar behavior when $\eta>0$.
This motivates us to compute the next-order expansion. It turns out
that at first order and with $V=-U_0n$,
$(n,E)$ is solving the (rescaled) energy-transport equations
\begin{align}\label{1.n}
  \pa_t n &= \diver\bigg(\frac{W\na n}{n(1-\eta n)}\bigg), \\
  \pa_t W &= \frac{2d-1}{2d}\diver\bigg(\frac{\na W}{n(1-\eta n)}\bigg)
  - U\frac{W|\na n|^2}{n(1-\eta n)}. \label{1.w}
\end{align}
where $U=U_0/(2d\eps_0^2)$ and $W=1-UE$ is the ``reverted'' energy. 
The case $W=0$ corresponds to the maximal energy $E=1/U_0$.
Taking into account the periodic lattice structure, we solve \eqref{1.n}-\eqref{1.w} 
on the torus $\T^d$, together with the initial conditions
$n(0)=n^0$, $W(0)=W^0$ in $\T^d$.

The structure of the diffusion equation \eqref{1.n} is similar to \eqref{1.log},
but the diffusion coefficient contains $W$ as a factor, adding a degeneracy to
the singular logarithmic diffusion equation. It is an open problem whether 
this factor removes the unphysical behavior of the solution to \eqref{1.log} in 
$\R^d$. We avoid this problem by solving \eqref{1.n}-\eqref{1.w} in a bounded domain
and by looking for strictly positive particle densities. Is is another open 
problem to prove the existence of solutions to \eqref{1.n}-\eqref{1.w} 
in the whole space.

Because of the squared gradient term in \eqref{1.w}, the energy $E$ (or $W$)
is not conserved but the total energy $W_{\rm tot}=W-(U/2)n^2$.
In fact, in terms of $W_{\rm tot}$, the squared gradient term is eliminated,
\begin{equation}\label{1.wtot}
  \pa_t W_{\rm tot} = \diver\bigg(\frac{\na W}{2n(1-\eta n)}
	+ \frac{UW}{1-\eta n}\na n\bigg).
\end{equation}
Unfortunately, this formulation does not help for the analysis since the
treatment of $\pa_t(n^2)=2n\pa_t n$ is delicate as $\pa_t n$ lies in the dual space 
$H^1(\T^d)'$ but $n$ is generally not an element of $H^1(\T^d)$ because of
the degeneracy (we have only $W^{1/2}\na n\in L^2(\T^d)$).

The analysis of system \eqref{1.n}-\eqref{1.w} is very challenging since
the first equation is degenerate in $W$, and the second equation contains
a quadratic gradient term. In the literature, there exist existence results
for degenerate equations with quadratic gradient terms \cite{DGLS06,GiMa08}, but
the degeneracy is of porous-medium type. A more complex degeneracy was
investigated in \cite{Cro12}. In our case, the degeneracy comes from another
variable, which is much more delicate to analyze. 

Related problems appear in semiconductor energy-transport theory,
but only partial results have been obtained so far. Let us review
these results.
The existence of stationary solutions to \eqref{1.et1} with the current densities
$$
  J_n = -\na(n\theta) + n\na V, \quad J_E = -\kappa_0\na \theta+\frac52\theta J_n,
  \quad E = \frac32 n\theta, 
$$ 
close to the constant equilibrium has been shown in \cite{AlRo17}.
The idea is that in such a situation, the temperature $\theta$
is strictly positive which removes the degeneracy in the term $\na(n\theta)$.
The parabolic system was investigated 
in \cite{JPR13,LLS16}, and the global existence of weak solutions was shown
without any smallness condition but for a simplifed energy equation. 
Again, the idea was to
prove a uniform positivity bound for the temperature, which removes the
degeneracy. A more general result (but without electric potential) was achieved
in \cite{ZaJu15} for the system
$$
  \pa_t n = \Delta (\theta^\alpha n), \quad 
  \pa_t(\theta n) = \Delta(\theta^{\alpha+1}n) 
  + \frac{n}{\tau}(1-\theta)
$$
in a bounded domain, where $0<\alpha<1$.
The global existence of weak solutions to the corresponding
initial-boundary-value problem was proved. Again, the idea is a positivity bound
for $\theta$ but this bound required a nontrivial cut-off procedure and
several entropy estimates.
 
In this paper, we make a step forward in the analysis of nonlinear parabolic
systems with nonstandard degeneracies by solving \eqref{1.n}-\eqref{1.w} without
any positive lower bound for $W$. Since $W$ may vanish, we can expect a gradient
estimate for $n$ only on $\{W>0\}$. Although the quadratic gradient term also
possesses $W$ as a factor, the treatment of this term is highly delicate,
because of low time regularity. 
Therefore, we present a result only for a time-discrete version of 
\eqref{1.n}-\eqref{1.w}, namely for its implicit Euler approximation
\begin{align}
  \frac{1}{\dt}(n^k-n^{k-1}) &= \diver\bigg(\frac{W^k\na n^k}{n(1-\eta n^k)}\bigg), 
	\label{1.appn} \\
	\frac{1}{\dt}(W^k-W^{k-1}) &= \frac{2d-1}{2d}\diver\bigg(\frac{\na W^k}{n^k
	(1-\eta n^k)}\bigg) - U\frac{W^k|\na n^k|^2}{n^k(1-\eta n^k)} \label{1.appw}
\end{align}
for $x\in\T^d$, where $\dt>0$ and $(n^{k-1},W^{k-1})$ are given functions.
We show the existence of a weak solution $(n^k,W^k)$ satisfying $n^k\ge 0$,
$W^k\ge 0$ and $W^kn^k$, $W^k\in H^1(\T^d)$; see Theorem \ref{thm.ex}.
In one space dimension and under a smallness assumption on the variance
of $W^{k-1}$ and $n^{k-1}$, the strict positivity of $W^k$ can be proved;
see Theorem \ref{thm.ex2}.

The existence proof is based on the solution of a regularized and truncated
problem by means of the Leray-Schauder fixed-point theorem. Standard elliptic
estimates provide bounds uniform in the approximation parameters. The key step
is the proof of the strong convergence of the gradient of the particle density.
For this, we show a general result for degenerate elliptic problems; see
Proposition \ref{prop.conv}. This result seems to be new. Standard results in
the literature need the ellipticity of the differential operator \cite{BoMu92}.
Unfortunately, we are not able to perform the limit $\triangle t\to 0$ since
some estimates in the proof of Proposition \ref{prop.conv} are not uniform 
in $\triangle t$; also see Remark \ref{rem.comm} for a discussion.

\subsection*{Numerical simulations}

The time-discrete system \eqref{1.appn}-\eqref{1.appw} is discretized by
finite differences in one space dimension and solved in an semi-implicit way. 
The large-time behavior exhibits an interesting phenomenon. If the initial 
energy $W^0$ is constant and sufficiently large, the solution $(n(t),W(t))$
converges to a constant steady state. However, if the constant $W^0$ is too
small, the stationary particle density is nonconstant. In both cases, the
time decay is exponential fast, but the decay rate becomes smaller for smaller
constants $W^0$ since the diffusion coefficient in \eqref{1.n} is smaller too.


\bigskip\noindent
The paper is organized as follows. Section \ref{sec.deriv} is devoted to the
formal derivation of the general energy-transport model and its entropy
structure, similar to the semiconductor case \cite{BeDe96}.
The high-temperature expansion is performed in Section \ref{sec.high},
leading to the energy-transport system \eqref{1.n}-\eqref{1.w}.
The strong convergence of the gradients
is shown in Section \ref{sec.conv}. In Section \ref{sec.ex} the existence
result is stated and proved. The numerical simulations are presented in
Section \ref{sec.num}, and the Appendix is concerned with the calculation
of some integrals involving the velocity $u(p)$ and energy $\eps(p)$.


\section{Formal derivation and entropy structure}\label{sec.deriv}

\subsection{Derivation from a Boltzmann equation}

We consider the following semiclassical Boltzmann transport equation 
for the distribution function $f(x,p,t)$ in the diffusive scaling:
\begin{equation}\label{2.be}
  \alpha\pa_t f_\alpha + u\cdot\na_x f_\alpha + \na V_\alpha\cdot
	\na_p f_\alpha = \frac{1}{\alpha}Q_\alpha(f_\alpha), \quad 
	(x,p)\in\R^{d}\times\T^d,\ t>0,
\end{equation}
where $\alpha>0$ is the Knudsen number \cite{BeDe96},
$(x,p)$ are the phase-space variables (space and crystal momentum), 
and $t>0$ is the time. We recall that the velocity equals
$u(p)=\na_p \eps(p)$, where the energy $\eps(p)$ is given by \eqref{1.eps}.
The potential $V_\alpha$ is defined by $V_\alpha=-U_0 n_\alpha$.
In the physical literature \cite{SHR12}, the
collision operator $Q_\alpha$ is given by the relaxation-time approximation
$$
  Q_\alpha(f) = \frac{1}{\tau_\alpha}(\F_f-f),
$$
where the function $\F_f$ is determined by maximizing the free energy \eqref{2.h}
associated to \eqref{2.be} under the constraints
\begin{equation}\label{2.mec}
  \int_{\T^d}(\F_f-f)dp = 0, \quad \int_{\T^d}(\F_f-f)\eps(p)dp = 0,
\end{equation}
which express mass and energy conservation during scattering events.
The solution of this problem is given by 
$$
  \F_f(x,p,t) = \frac{1}{\eta + \exp(-\lambda_0(x,t)-\lambda_1(x,t)\eps(p))},
$$
where $\lambda_0$ and $\lambda_1$ are the Lagrange multipliers and $\eta\ge 0$
is a parameter which may take the values $\eta=0$ (Maxwell-Boltzmann statistics)
or $\eta=1$ (Fermi-Dirac statistics). The relaxation time $\tau_\alpha\ge 0$
generally depends on the particle density but at this point we do not need
to specifiy the dependence. 

We show the following result.

\begin{proposition}[Derivation]\label{prop.deriv}
Let $f_\alpha$ be a (smooth) solution to the Boltzmann equation \eqref{2.be}.
We assume that the formal limits $f=\lim_{\alpha\to 0}f_\alpha$,
$g=\lim_{\alpha\to 0}(f_\alpha-\F_{f_\alpha})/\alpha$, and
$\tau=\lim_{\alpha\to 0}\tau_\alpha$ exist. Then the 
particle and energy densities
$$
  n = n[\F_f] = \int_{\T^d}\F_f dp, \quad E = E[\F_f] = \int_{\T^d}\F_f\eps(p)dp
$$
are solutions to \eqref{1.et1}-\eqref{1.et2},
and the diffusion coefficients $D_{ij}=(D_{ij}^{k\ell})\in\R^{d\times d}$
are defined by 
$$
  D_{ij}^{k\ell} = \tau\int_{\T^d}u_k u_\ell\F_f(1-\eta\F_f)\eps(p)^{i+j}dp, \quad
	i,j=0,1,\ k,\ell=1,\ldots,d.
$$
\end{proposition}

The proof of the proposition is similar to those of Propositions 1 and 2 in
\cite{JKP11}. For the convenience of the reader, we present the (short)
proof.

\begin{proof}
To derive the balance equations, we multiply the
Boltzmann equation \eqref{2.be} by $1$ and $\eps$, respectively, and integrate over
$\T^d$:
\begin{align}
  & \pa_t n[f_\alpha] + \frac{1}{\alpha}\mathrm{div}_x\int_{\T^d}uf_\alpha dp
 = 0, \label{2.m1} \\
	& \pa_t E[f_\alpha] + \frac{1}{\alpha}\mathrm{div}_x\int_{\T^d}
	\eps u f_\alpha dp - \frac{1}{\alpha}\na_x V_\alpha\cdot\int_{\T^d}uf_\alpha dp = 0.
	\label{2.m2}
\end{align}
The integrals involving the collision operator vanish in view of mass and
energy conservation; see \eqref{2.mec}.
We have integrated by parts in the last integral on the left-hand side of
\eqref{2.m2}. Next, we insert the Chapman-Enskog expansion 
$f_\alpha=\F_{f_\alpha}+\alpha g_\alpha$ (which in fact defines $g_\alpha$)
in \eqref{2.m1}-\eqref{2.m2} and observe that the function
$p\mapsto u(p)\eps(p)^j\F_{f_\alpha}(p)$ is odd for any $j\in\N_0$ such that
its integral over $\T^d$ vanishes. This leads to
\begin{align*}
  & \pa_t n[\F_{f_\alpha}] + \alpha\pa_t n[g_\alpha]
	+ \mathrm{div}_x\int_{\T^d}ug_\alpha dp = 0, \\
	& \pa_t E[\F_{f_\alpha}] + \alpha\pa_t E[g_\alpha]
	+  \mathrm{div}_x\int_{\T^d}u\eps g_\alpha dp 
	- \na_x V\cdot\int_{\T^d}ug_\alpha dp = 0.
\end{align*}
Passing to the formal limit $\alpha\to 0$ gives the balance equations 
\eqref{1.et1} with 
\begin{equation}\label{2.auxJ}
  J_n = \int_{\T^d} ugdp, \quad J_E = \int_{\T^d} u\eps gdp.
\end{equation}

To specify the current densities, we insert the Chapman-Enskog expansion
in \eqref{2.be},
$$
  \alpha\pa_t(\F_{f_\alpha}+\alpha g_\alpha) 
	+ u\cdot\na_x(\F_{f_\alpha}+\alpha g_\alpha)
	+ \na_x V\cdot\na_p(\F_{f_\alpha}+\alpha g_\alpha) = -\frac{g_\alpha}{\tau_\alpha},
$$
and perform the formal limit $\alpha\to 0$,
\begin{equation}\label{2.auxg}
  u\cdot\na_x\F_f + \na_x V\cdot\na_p \F_f = -\frac{g}{\tau}.
\end{equation}
A straightforward computation shows that
$$
  \na_x\F_f = \F_f(1-\eta\F_f)(\na_x\lambda_0 + \eps\na_x\lambda_1), \quad
	\na_p\F_f = \F_f(1-\eta\F_f)u\lambda_1,
$$
and inserting this into \eqref{2.auxg} gives an explicit expression for $g$:
$$
  g = -\tau \F_f(1-\eta\F_f)\big(u\cdot\na_x\lambda_0 + \eps u\cdot\na_x\lambda_1
	+ \lambda_1\na_x V\cdot u\big).
$$
Therefore, the current densities \eqref{2.auxJ} lead to \eqref{1.et2}.
This finishes the proof.
\end{proof}

In the following we write $\F_f=\F(\lambda)$, where
\begin{equation}\label{2.F}
  \F(\lambda) = \frac{1}{\eta + \exp(-\lambda_0-\lambda_1\eps(p))}, \quad
	\lambda=(\lambda_0,\lambda_1)\in\R^2,\ p\in\T^d.
\end{equation}

\begin{proposition}[Diffusion matrix]\label{prop.D}
The diffusion matrix $\DD=(D_{ij})\in\R^{2d\times 2d}$ is symmetric and
positive definite.
\end{proposition}

\begin{proof}
The proof is similar to Proposition 3 in \cite{JKP11}. Let $z=(w,y)\in\R^{2d}$
with $w$, $y\in\R^d$. Then
\begin{align*}
  z^\top\DD z 
	&= w^\top D_{00}w + 2w^\top D_{01}y + y^\top D_{11}y \\
	&= \int_{\T^d}\F(1-\eta\F)\sum_{i=1}^d(u_iw_i+\eps u_i y_i)\sum_{j=1}^d
	(u_jw_j+\eps u_j y_j)dp \\
	&= \int_{\T^d}\F(1-\eta\F)\sum_{i=1}^d\big|u_i(w_i+\eps y_i)\big|^2 dp \ge 0.
\end{align*}
Since $D_{ij}^{k\ell}$ is symmetric in $(i,j)$ and $(k,\ell)$, the symmetry
of $\DD$ is clear.
\end{proof}


\subsection{Entropy structure}\label{sec.ent}

The entropy structure of \eqref{1.et1}-\eqref{1.et2} follows from the abstract
framework presented in \cite{JKP11}. In the following, we make this
framework explicit. First, we introduce the entropy
\begin{align*}
  H(t) &= \int_{\R^d} h(\lambda)dx, \quad\mbox{where} \\
	h(\lambda) &= \int_{\T^d}\big(
	\F\log\F + \eta^{-1}(1-\eta\F)\log(1-\eta\F)\big)dp.
\end{align*}
The entropy density $h$ can be reformulated as
\begin{align}
  h(\lambda) &= \int_{\T^d}\bigg(\F\log\frac{\F}{1-\eta\F} 
	- \frac{1}{\eta}\log\frac{1}{1-\eta\F}\bigg)dp \nonumber \\
	&= \int_{\T^d}\big(\F(\lambda_0+\lambda_1\eps) - \frac{1}{\eta}\log
	(1+\eta e^{\lambda_0+\lambda_1\eps})\big)dp \nonumber \\
	&= n\lambda_0 + E\lambda_1  - \frac{1}{\eta}\int_{\T^d}\log
	(1+\eta e^{\lambda_0+\lambda_1\eps})dp. \label{2.h}
\end{align}
The following result shows that the entropy is nonincreasing in time.

\begin{proposition}[Entropy structure]\label{prop.ent}
It holds that
\begin{equation}\label{2.dHdt}
  \frac{dH}{dt} = -\int_{\R^d}\sum_{i,j=0}^1\na\mu_i^\top L_{ij}\na\mu_j dx \le 0,
\end{equation}
where $\mu_0=\lambda_0+\lambda_1 V$ and $\mu_1=\lambda_1$ are the so-called 
dual entropy variables and 
\begin{equation}\label{2.L}
  L_{00} = D_{00}, \quad L_{01} = L_{10} = D_{01}-D_{00}V, \quad
	L_{11} = D_{11} - 2D_{01}V + D_{00}V^2.
\end{equation}
\end{proposition}

\begin{proof}
Identity \eqref{2.h} implies that
$$
  \frac{\pa h}{\pa\lambda_i} 
	= \frac{\pa n}{\pa\lambda_i}\lambda_0 + \frac{\pa E}{\pa\lambda_i}\lambda_1,
	\quad i=0,1,
$$
and consequently,
\begin{align*}
  \pa_t h(\lambda) 
	&= \frac{\pa h}{\pa\lambda_0}\pa_t\lambda_0
	+ \frac{\pa h}{\pa\lambda_1}\pa_t\lambda_1  \\
	&= \lambda_0\bigg(\frac{\pa n}{\pa\lambda_0}\pa_t\lambda_0
	+ \frac{\pa n}{\pa\lambda_1}\pa_t\lambda_1\bigg)
	+ \lambda_1\bigg(\frac{\pa E}{\pa\lambda_0}\pa_t\lambda_0 
	+ \frac{\pa E}{\pa\lambda_1}\pa_t\lambda_1\bigg) \\
	&= \lambda_0\pa_t n+\lambda_1\pa_t E.
\end{align*}
Therefore, using \eqref{1.et1}-\eqref{1.et2} and integration by parts,
\begin{align*}
  \frac{dH}{dt} &= \int_{\R^d}\pa_t h(\lambda) dx
	= \int_{\R^d}\big(J_n\cdot\na\lambda_0 + J_E\cdot\na\lambda_1
	+ \na V\cdot J_n \lambda_1\big)dx \\
	&= -\int_{\R^d}\big(D_{00}|\na\mu_0|^2 + 2(D_{01}-D_{00}V)\na\mu_0\cdot\na\mu_1 \\
	&\phantom{xx}{}+ (D_{11}-2D_{01}V+D_{00}V^2)|\na\mu_1|^2\big)dx,
\end{align*}
which proves the identity in \eqref{2.dHdt}. Using the positive definiteness
of $\DD$, a computation shows that $(L_{ij})$ is positive definite too, and
the inequality in \eqref{2.dHdt} follows.
\end{proof}


\subsection{Singularities and degeneracies in the energy-transport system}

We denote by $ n$ and $ E$ the particle and energy densities
depending on the dual entropy variable $\mu=(\mu_0,\mu_1)
=(\lambda_0+V\lambda_1,\lambda_1)$. We have the (implicit) formulation
\begin{align*}
   n(\mu) &= \int_{\T^d}\frac{dp}{\eta + \exp(-\mu_0+U_0 n(\mu)
  - \mu_1\eps(p))}, \\
   E(\mu) &= \int_{\T^d}\frac{\eps(p)dp}{\eta + \exp(-\mu_0+U_0 n(\mu)
	- \mu_1\eps(p))}.
\end{align*}

\begin{lemma}\label{lem.det}
Let $\omega_i(\mu):=\int_{\T^d}\F(1-\eta\F)\eps(p)^i dp$, $i\in\N_0$. 
Then 
$$
  \det\frac{\pa(n, E)}{\pa\mu}
  = \frac{\omega_0\omega_2-\omega_1^2}{1-U_0\mu_1\omega_0}.
$$
\end{lemma}

\begin{proof}
We differentiate
$$
  \frac{\pa n}{\pa\mu_0} 
  = \bigg(1+U_0\mu_0\frac{\pa n}{\pa\mu_0}\bigg)\omega_0, \quad
  \frac{\pa n}{\pa\mu_1} 
  = U n + U_0\mu_1\omega_0\frac{\pa n}{\pa\mu_0} + \omega_1.
$$
This gives after a rearrangement
$$
  \frac{\pa n}{\pa\mu_0} = \frac{\omega_0}{1-U_0\mu_1\omega_0}, \quad
  \frac{\pa n}{\pa\mu_1} 
  = \frac{U n\omega_0+\omega_1}{1-U_0\mu_1\omega_0}.
$$
In a similar way, we obtain 
$$
  \frac{\pa E}{\pa\mu_0} = \frac{\omega_1}{1-U_0\mu_1\omega_0}, 
	\quad\frac{\pa E}{\pa\mu_1} 
  = \frac{U n\omega_1 + U_0\mu_1(\omega_1
  	-\omega_0\omega_2)+\omega_2}{1-U_0\mu_1\omega_0},
$$
and with
$$
  \det\frac{\pa(n, E)}{\pa\mu}
  =  \frac{\pa n}{\pa\mu_0} \frac{\pa E}{\pa\mu_1}
  -  \frac{\pa n}{\pa\mu_1} \frac{\pa E}{\pa\mu_0}
  = \frac{(\omega_0\omega_2-\omega_1^2)(1-U_0\mu_1\omega_0)}{(1-U_0\mu_1\omega_0)^2},
$$
the conclusion follows. 
\end{proof}

Since $\mu_1$ can be positive and $\omega_0>0$, the expression $1-U_0\mu_1\omega_0$
may vanish,
so the determinant of $\pa(n,E)/\pa\mu$ may be not finite. Moreover, 
the numerator of the determinant may vanish, and the function $\mu\mapsto(n,E)$
may be not invertible. This is made more explicit in the following remark.

\begin{remark}[Case $\mu=0$]\rm
In the Maxwell-Boltzmann case, we can make the numerator
of the determinant in Lemma \ref{lem.det}
explicit. Indeed, it is clear that $\omega_0= n$ and $\omega_1= E$.
For the computation of $\omega_2$, we observe first that
\begin{align}
  n &= \int_{\T^d}\F dp = \exp(\mu_0-U_0 n\mu_1)\prod_{k=1}^d\int_{\T}
	\exp\big(-2\eps_0\cos(2\pi p_k)\big)dp_k \nonumber \\
	&= \exp(\mu_0-U_0 n\mu_1)I_0^d, \label{2.nI} \\
	E &= \int_{\T^d}\F\eps dp = -2\eps_0\exp(\mu_0-U_0 n\mu_1)\sum_{i=1}^d\prod_{k\neq i}^d
	\int_{\T}\exp\big(-2\eps_0\cos(2\pi p_k)v)dp_k \nonumber \\
	&\phantom{xx}{}\times \int_{\T}
	\exp\big(-2\eps_0\cos(2\pi p_i)\big)\cos(2\pi p_i)dp_i \nonumber \\
	&= -2d\eps_0\exp(\mu_0-U_0 n\mu_1)I_0^{d-1}I_1, \label{2.EI}
\end{align}
where, by symmetry,
$$
  I_0 := \int_{\T}\exp\big(-2\eps_0\cos(2\pi p_1)\big)dp_1, \quad
	I_1 := \int_{\T}\exp\big(-2\eps_0\cos(2\pi p_1)\big)\cos(2\pi p_1)dp_1.
$$
Now, we have
$$
  \omega_2
  = 4\eps_0^2\sum_{i=1}^d\int_{\T^d}\cos^2(2\pi p_i)\F dp 
  + 8\eps_0^2\sum_{i=1}^d\sum_{j=1,\,j\neq i}^d
  \int_{\T^d}\cos(2\pi p_i)\cos(2\pi p_j)\F dp.
$$
Expanding the exponentials in $\F$ and using \eqref{2.nI}-\eqref{2.EI}, we find that
$$
  \int_{\T^d}\cos(2\pi p_i)\cos(2\pi p_j)\F dp
	= \exp(\mu_0-U_0 n\mu_1) I_0^{d-2}I_1^2 
	= \frac{E^2}{(2d\eps_0)^2n},
$$
and this expression is independent of $i\neq j$.
For $i=j$, we use the identity
$\sin(2\pi p_i)\F=(\pa\F/\pa p_i)/(4\pi\eps_0\mu_1)$ and integration by parts:
\begin{align*}
  \int_{\T^d}\cos^2(2\pi p_i)\F dp
  &= \int_{\T^d}(1-\sin^2(2\pi p_i))\F dp
  = n - \int_{\T^d}\frac{\sin(2\pi p_i)}{4\pi\eps_0\mu_1}
  \frac{\pa\F}{\pa p_i}dp \\
  &= n + \int_{\T^d}\frac{\cos(2\pi p_i)}{2\eps_0\mu_1}\F dp. 
\end{align*}
Summing over $i=1,\ldots,n$ gives
$$
  4\eps_0^2\sum_{i=1}^d  \int_{\T^d}\cos^2(2\pi p_i)\F dp
  = 4\eps_0^2 dn - \frac{E}{\mu_1}.
$$
We conclude that
$\omega_2 = 4\eps_0^2dn - E/\mu_1 + (d-1)E^2/(dn)$ and consequently, 
$$
  \omega_0\omega_2-\omega_1^2 = 4\eps_0^2 dn - \frac{En}{\mu_1} - \frac{E^2}{d}.
$$
For certain values of $\mu_1$ or $(n,E)$, this expression may vanish such that
$\det\pa(n,E)/\pa\mu=0$ at these values.
This shows that the relation between $(n,E)$ and $\mu$ needs to treated with care.
\qed
\end{remark}

\begin{remark}[Degeneracy in the entropy production]\rm
A tedious computation, detailed in \cite[Chapter 8.4]{Bra17}, shows that the entropy
production can be written as
$$
  \sum_{i,j=0}^1\int_{\R^d}\na\mu_i^\top L_{ij}\na\mu_j dx
  = g_1(\lambda)(1-U_0\mu_1\omega_0)|\na n|^2 
  + g_2(\lambda)|g_3(\lambda)\na n-\na E|^2,
$$
where $g_i(\lambda)$, $i=1,2,3$, are functions depending on 
$\omega_i$, defined in Lemma \ref{lem.det}, and on
$$
  \Gamma_i = \int_{\T^d}\eps(p)^i|\na\eps|^2\F(1-\eta\F)dp, \quad i=0,1,2.
$$
The above formula shows that we lose the gradient estimate if
$1-U_0\mu_1\omega_0=0$.
\qed
\end{remark}

\begin{remark}[Comparison with the semiconductor case]\rm
For the semiconductor energy-transport equations in the parabolic band
approximation, we do not face the singularities and 
degeneracies occuring in the model for optical lattices.	Indeed,
let the potential $V$ be given (to simplify).
According to Example 6.8 in \cite{Jue09}, we have
$$
  n = \mu_1^{-3/2}\exp(\mu_0+\mu_1 V), \quad
  E = \frac32\mu_1^{-5/2}\exp(\mu_0+\mu_1 V).
$$
Then
$$
  \det\frac{\pa(n,E)}{\pa\mu} = \det\begin{pmatrix}
  n & nV-E \\ E & -5E/(2\mu_1)+EV
  \end{pmatrix}
  = -\frac23E^2,
$$
which is nonzero as long as $E>0$. Furthermore, by Remark 8.12 in \cite{Jue09},
it holds that $\omega_0=n$, $\omega_1=E$, and $\omega_2=5E^2/(3n)$, and so
$$
  \omega_0\omega_2-\omega_1^2 = \frac23 E^2.
$$
This expression is degenerate only at the boundary of the domain of definition 
(i.e.\ at $E=0$). Often, such kind of degeneracies may be handled; an important
example is the porous-medium equation. 
In the case of optical lattices, the degeneracy may occur in the interior of
the domain of definition, which is much more delicate.
\qed
\end{remark}


\section{High-temperature expansion}\label{sec.high}

The Lagrange multiplier $\lambda_1$ is interpreted as the negative inverse temperature,
so high temperatures correspond to small values of $|\lambda_1|$. In this
section, we perform a high-temperature expansion of \eqref{1.et1}-\eqref{1.et2},
i.e., we expand $\F(\lambda)$ around $(\lambda_0,0)$ for small $|\lambda_1|$
up to first order. Our ansatz is
\begin{align}
  \F(\lambda) &= \F(\lambda_0,0) + \frac{\pa \F}{\pa\lambda_1}(\lambda_0,0)\lambda_1
	+ O(\lambda_1^2) \nonumber \\
	&= \F(\lambda_0,0) + \eps\F(\lambda_0,0)\big(1-\eta\F(\lambda_0,0)\big)\lambda_1
	+ O(\lambda_1^2). \label{3.F}
\end{align}

\subsection{Zeroth-order expansion}

At zeroth-order, we have by \eqref{2.auxJ}, \eqref{2.auxg}, and using
formula \eqref{2.uu} from the appendix,
\begin{align*}
  n &= \int_{\T^d}\F(\lambda_0,0)dp + O(\lambda_1) 
	= \F(\lambda_0,0) + O(\lambda_1), \\
	J_n &= -\tau\int_{\T^d} \big(u\otimes u\na_x\F(\lambda_0,0) 
	+ u\na_x V\cdot\na_p\F(\lambda_0,0)\big)dp \\
	&= -\tau\int_{\T^d}(u\otimes u)n dp + O(\lambda_1) 
	= -\frac{\tau}{2}(4\pi\eps_0)^2\na n + O(\lambda_1).
\end{align*}
Therefore, up to order $O(\lambda_1)$, we infer that
$$
  \pa_t n = \frac12(4\pi\eps_0)^2\diver(\tau\na n), \quad x\in\R^d,\ t>0,
$$
At high temperature, the relaxation time depends on the particle density in a 
nonlinear way, $\tau=\tau_0/(n(1-\eta n))$ \cite{MRR11}. At low densities, i.e.\
$\eta=0$, we obtain the logarithmic diffusion equation
$$
  \pa_t n = \eps_1\Delta\log n, \quad t>0, \quad n(0,\cdot)=n_0\quad\mbox{in }\R^d,
$$
where $\eps_1=\frac12\tau_0(4\pi\eps_0)^2$. We already mentioned in the introduction
that the (smooth) solution to this equation in two space dimensions loses mass, 
which is unphysical. Therefore, we compute the next-order expansion.


\subsection{First-order expansion}

We calculate, using \eqref{2.eps2},
\begin{align*}
  n &= \int_{\T^d}\big(\F(\lambda_0,0)+\eps(\F(1-\eta\F))(\lambda_0,0)
	\lambda_1\big)dp + O(\lambda_1^2) \\
	&= \F(\lambda_0,0) + (\F(1-\eta\F))(\lambda_0,0)\lambda_1\int_{\T^d}\eps(p)dp
	+ O(\lambda_1^2) = \F(\lambda_0,0) + O(\lambda_1^2), \\
  E &= \F(\lambda_0,0)\int_{\T^d}\eps(p)dp + (\F(1-\eta\F))(\lambda_0,0)\lambda_1
	\int_{\T^d}\eps(p)^2dp + O(\lambda_1^2) \\
	&= 2d\eps_0^2(\F(1-\eta\F))(\lambda_0,0)\lambda_1 + O(\lambda_1^2)
	= 2d\eps_0^2(\F(1-\eta\F))(\lambda_0,0)\lambda_1 + O(\lambda_1^2).
\end{align*}
Therefore, by \eqref{3.F},
$$
  \F(\lambda) = n + \eps\frac{E}{2d\eps_0^2} + O(\lambda_1^2),
$$
and \eqref{2.auxg} and $\na_p\eps=u$ give, up to order $O(\lambda_1^2)$,
\begin{align*}
  g &= -\tau\big(u\cdot\na_x\F(\lambda) + \na_x V\cdot\na_p\F(\lambda)\big) \\
	&= -\tau u\cdot\bigg(\na_x n + \frac{E}{2d\eps_0^2}\na_x V\bigg) 
	- \tau\eps u\cdot\frac{\na_x E}{2d\eps_0^2}.
\end{align*}
Then, by \eqref{2.auxJ}-\eqref{2.auxg} and taking into account 
\eqref{2.uu}--\eqref{2.eps2uu}, we infer that, again up to first order,
\begin{align*}
	J_n &= -\tau\int_{\T^d} u\otimes udp\na_x n
	- \tau\int_{\T^d}u\otimes udp\na_x V\frac{E}{d\eps_1}
	= -8\pi^2\tau\eps_0^2\na_x n - \frac{4\pi^2}{d}\tau E\na_x V, \\
	J_E &= -\frac{\tau}{2d\eps_0^2}\int_{\T^d}\eps^2(u\otimes u)dp \na_x E 
	= -4\pi^2\frac{2d-1}{d}\eps_0^2\na_x E.
\end{align*}
Therefore, the first-order expansion leads to
\begin{align}
  & \pa_t n = 8\pi^2\eps_0^2\diver\bigg(\tau\na n 
	+ \frac{\tau}{2d\eps_0^2}E\na V\bigg), \label{2.et1} \\
	& \pa_t E = 8\pi^2\eps_0^2\frac{2d-1}{2d}\diver(\tau\na E) 
	- 8\pi^2\eps_0^2\tau\na V\cdot\bigg(\na n+\frac{E}{2d\eps_0^2}\na V\bigg). 
	\label{2.et2}
\end{align}
We rescale the time by
$t_s=(8\pi^2\eps_0^2\tau_0)t$ and introduce $U=U_0/(2d\eps_0^2)$ and
$W=1-UE$. Then, writing
again $t$ instead of $t_s$, system \eqref{2.et1}-\eqref{2.et2} becomes
\begin{equation}\label{2.eq}
  \pa_t n = \diver\bigg(\frac{W\na n}{n(1-\eta n)}\bigg), \quad
  \pa_t W = \frac{2d-1}{2d}\diver\bigg(\frac{\na W}{n(1-\eta n)}\bigg)
  - U\frac{W|\na n|^2}{n(1-\eta n)}.
\end{equation}
The existence of weak solutions to a time-discrete version of \eqref{2.eq},
together with periodic boundary conditions, is shown in Section \ref{sec.ex}.


\section{A strong convergence result for the gradient}\label{sec.conv}

The key tool of the existence analysis of Section \ref{sec.ex} is the following
result on the strong convergence of the gradients of certain approximate
solutions for the following equation. 
Let $\dt>0$, $\overline{n}\in L^\infty(\Omega)$, 
$y\in L^\infty(\Omega)\cap H^1(\Omega)$, and 
$\Psi\in C^1(\R)$. We consider the equation 
\begin{equation}\label{auxeq2}
	\frac{1}{\dt}(n-\overline{n}) = \diver\big(\na(y\Psi(n))-\Psi(n)\na y\big)
	\end{equation}	
for $n\in L^\infty(\Omega)$ such that $y\Psi(n)\in H^1(\Omega)$.

\begin{proposition}\label{prop.conv}
Let $\Omega\subset\R^d$ be a bounded domain, $\eps>0$, $\triangle t>0$, let
$\overline{n}\in L^\infty(\Omega)$ be such that $\overline{n}\ge 0$ in $\Omega$, 
and let $\Psi\in C^1(\R)$ satisfy $\Psi'>0$.
Let $(y_\eps)$ be a bounded sequence in $H^1(\Omega)$ satisfying 
$y_\eps\ge C(\eps)>0$ for some $C(\eps)>0$, 
$y_\eps\to y$ strongly in $L^2(\Omega)$ and weakly in 
$H^1(\Omega)$ as $\eps\to 0$.
Furthermore, let $n_\eps\in L^2(\Omega)$ with $\Psi(n_\eps)\in H^1(\Omega)$  
be a weak solution to
\begin{equation}\label{auxeq}
  \frac{1}{\dt}\int_\Omega(n_\eps-\overline{n})\phi dx
	+ \int_\Omega y_\eps\na \Psi(n_\eps)\cdot\na\phi dx = 0
\end{equation}
for all $\phi\in H^1(\Omega)$. Then there exist a function $n\in L^\infty(\Omega)$ 
such that $y\Psi(n)\in H^1(\Omega)$, being a weak solution of \eqref{auxeq2}, 
and a subsequence of $(n_\eps)$, which is not relabeled, such that, as $\eps\to 0$,
\begin{align*}
  y^{1/2}y_\eps^{1/2}\na \Psi(n_\eps) \to \na(y\Psi(n))-\Psi(n)\na y 
	& \quad\mbox{strongly in }L^2(\Omega), \\
	\mathrm{1}_{\{y>0\}}(n_\eps-n)\to 0 &\quad\mbox{strongly in }L^2(\Omega).
\end{align*}
\end{proposition}

\begin{proof} 
{\em Step 1.}
First, we derive some uniform bounds. Set $M=\|\overline{n}\|_{L^\infty(\Omega)}$. 
Taking $(\Psi(n_\eps)-\Psi(M))_+=\max\{0,\Psi(n_\eps)-\Psi(M)\}$ as a test function 
in \eqref{auxeq}, we find that
\begin{align*}
  \frac{1}{\dt}\int_\Omega & \big((n_\eps-M)-(\overline{n}-M)\big)
	(\Psi(n_\eps)-\Psi(M))_+ dx \\
	&{}+ \int_\Omega y_\eps\na \Psi(n_\eps)\cdot\na(\Psi(n_\eps)-\Psi(M))_+ 
	= 0.
\end{align*}
Since $-(\overline{n}-M)(\Psi(n_\eps)-\Psi(M))_+\ge 0$, it follows that
$$
  \frac{1}{\dt}\int_\Omega(n_\eps-M)_+(\Psi(n_\eps)-\Psi(M))_+ dx 
  + \int_\Omega y_\eps|\na(\Psi(n_\eps)-\Psi(M))_+|^2 dx \le 0
$$
and hence, $n_\eps\le M$ in $\Omega$. In a similar way, using 
$\Psi(n_\eps)_-=\min\{0,\Psi(n_\eps)\}$ as a test function and using 
$\overline{n}\ge 0$, we infer that $n_\eps\ge 0$. 
This shows that $(n_\eps)$ is bounded in 
$L^\infty(\Omega)$. Hence, there exists a subsequence which is not relabeled
such that, as $\eps\to 0$,
\begin{equation}\label{auxn}
  n_\eps\rightharpoonup^* n\quad \mbox{weakly* in }L^\infty(\Omega)
\end{equation}
for some function $n\in L^\infty(\Omega)$. Since $\Phi'$ is positive and continuous 
on $\R$, the boundedness of $(n_\eps)$ in $L^\infty(\Omega)$ implies 
that there exists a constant $C>0$  such that
$1/C\leq \Psi'(n_\eps)\leq C$ for all $\eps>0$.
Next, we choose the test function $\phi=\Psi(n_\eps)$
in \eqref{auxeq} and use $\Psi(n_\eps)\leq \Psi(M)$ to find that
$$
  \frac{1}{\dt}\int_\Omega n_\eps\Psi(n_\eps) dx 
	+ \int_\Omega y_\eps|\na\Psi(n_\eps)|^2 dx
	\le \frac{\Psi(M)}{\dt}\int_\Omega \overline{n} dx.
$$
We deduce that $(y_\eps^{1/2}\na\Psi(n_\eps))$ is bounded 
in $L^2(\Omega)$ and, for a subsequence, 
\begin{equation}\label{auxna}
  y_\eps^{1/2}\na\Psi(n_\eps) \rightharpoonup \xi\quad\mbox{weakly in }L^2(\Omega)
\end{equation}
for some function $\xi\in L^2(\Omega)$. Now, let $\phi$ be a smooth test function. 
Then we can take the limit, for a subsequence, in \eqref{auxeq} and obtain
\begin{equation}\label{auxeq_im_Beweis}
  \frac{1}{\dt}\int_\Omega(n-\overline{n})\phi dx
  + \int_\Omega y^{1/2}\xi\cdot\na\phi dx = 0.
\end{equation}
Note that this equation holds also for all $\phi\in H^1(\Omega)$.

{\em Step 2.}
As $(y_\eps)$ is strongly converging in $L^2(\Omega)$,
we deduce from \eqref{auxn} that $y_\eps^{1/2}(n_\eps-n)\rightharpoonup 0$ 
weakly in $L^2(\Omega)$.
We claim that this convergence is even strong. Indeed,
the sequence 
$$
  \na(y_\eps n_\eps^2) = 2\frac{y_\eps^{1/2}n_\eps}{\Psi'(n_\eps)}
	y_\eps^{1/2}\na \Psi(n_\eps) + n_\eps^2\na y_\eps
$$
is uniformly bounded in $L^2(\Omega)$.
Thus, $(y_\eps n_\eps^2)$ is bounded in $H^1(\Omega)$ and by compactness, for a
subsequence, $y_\eps n_\eps^2\to \zeta\ge 0$ strongly in $L^2(\Omega)$ or
$y_\eps^{1/2}n_\eps\to \zeta^{1/2}$ strongly in $L^4(\Omega)$. The
strong convergence of $(y_\eps^{1/2})$ in $L^2(\Omega)$ and the weak* convergence
of $(n_\eps)$ in $L^\infty(\Omega)$ imply that $y_\eps^{1/2}n_\eps\rightharpoonup
y^{1/2}n$ weakly in $L^2(\Omega)$. Therefore, $\zeta^{1/2}=y^{1/2}n$ and
\begin{equation}\label{auxyn}
  y_\eps^{1/2}(n_\eps-n)\to 0\quad\mbox{strongly in }L^2(\Omega).
\end{equation}
This proves the claim. 

{\em Step 3.} The next goal is to show that
\begin{equation}\label{auxnn}
  \mathrm{1}_{\{y>0\}}(n_\eps-n)\to 0\quad\mbox{strongly in }L^2(\Omega).
\end{equation}
Taking into account \eqref{auxyn}, 
the strong convergence of $(y_\eps^{1/2})$ in $L^2(\Omega)$,
and the $L^\infty$ bound for $(n_\eps)$, it follows that
$$
  \|y^{1/2}(n_\eps-n)\|_{L^2(\Omega)}
	\le \|y_\eps^{1/2}(n_\eps-n)\|_{L^2(\Omega)}
	+ \|y^{1/2}-y_\eps^{1/2}\|_{L^2(\Omega)}\|n_\eps-n\|_{L^\infty(\Omega)}
$$
converges to zero. Thus, for a subsequence, $y^{1/2}(n_\eps-n)\to 0$ a.e.\
in $\Omega$ and consequently, $n_\eps-n\to 0$ a.e.\ in $\{y>0\}$.
Then the a.e.~pointwise convergence $\mathrm{1}_{\{y>0\}}(n_\eps-n)\to 0$
and the dominated convergence theorem show \eqref{auxnn}.

{\em Step 4.} We wish to identify $\xi$ in \eqref{auxna} and \eqref{auxeq_im_Beweis}. 
The bounds for $(\na y_\eps)$ and $(y_\eps\na\Psi(n_\eps))$ in $L^2(\Omega)$ show 
that $(\na(y_\eps\Psi(n_\eps)))$ is bounded in $L^2(\Omega)$ and so, 
$(y_\eps \Psi(n_\eps))$ is bounded in $H^1(\Omega)$. By compactness, for a subsequence,
$\na(y_\eps\Psi(n_\eps))\rightharpoonup \na(y\Psi(n))$ weakly in $L^2(\Omega)$ and
$y_\eps\Psi(n_\eps)\to y\Psi(n)$ strongly in $L^2(\Omega)$. We can identify the limit
since $y_\eps\to y$ strongly in $L^2(\Omega)$ and 
$\Psi(n_\eps)\rightharpoonup^* \theta$ weakly* in $L^\infty(\Omega)$ 
with $\theta=\Psi(n)$ in $\{y>0\}$ lead to 
$y_\eps \Psi(n_\eps)\rightharpoonup y\Psi(n)$ weakly in
$L^2(\Omega)$. Moreover, since $\Psi(n_\eps)\to \Psi(n)$ strongly in $L^2(\Omega)$ 
and $\na y_\eps\to\na y$ weakly in $L^2(\Omega)$, we infer that
\begin{equation}\label{auxyy}
	y_\eps\na \Psi(n_\eps) = \na(y_\eps \Psi(n_\eps))-\Psi(n_\eps)
	\na y_\eps\rightharpoonup\na(y \Psi(n))-\Psi(n)\na y 
	\quad\mbox{weakly in }L^2(\Omega).
\end{equation}
Here, we have used additionally that $(y_\eps\na \Psi(n_\eps))$ is bounded in 
$L^2(\Omega)$ and that $L^1(\Omega)$ is dense in $L^2(\Omega)$. 
Similarly as above, we deduce that
\begin{equation}\label{auxabschaetzung_gut_fuer_spaeter}
  \|(y^{1/2}-y_\eps^{1/2})y_\eps^{1/2}\na \Psi(n_\eps)\|_{L^1(\Omega)}
	\leq \|(y^{1/2}-y_\eps^{1/2})\|_{L^2(\Omega)}
	\|y_\eps^{1/2}\na \Psi(n_\eps)\|_{L^2(\Omega)}
\end{equation}
converges to zero. Therefore, $y^{1/2}\xi=\na(y \Psi(n))-\Psi(n)\na y$ in
$\Omega$.

{\em Step 5.} We obtain from \eqref{auxyy} that 
$$
  y^{1/2}\big(y_\eps\na \Psi(n_\eps)-y^{1/2}\xi\big)\rightharpoonup 0 
	\quad\mbox{weakly in }L^2(\Omega)
$$
and consequently,
\begin{align}\label{auxc1}
	\int_\Omega y_\eps\xi\cdot\big(y^{1/2}\na \Psi(n_\eps)-\xi\big)dx
	&= \int_\Omega \xi\cdot y^{1/2}\big(y_\eps\na \Psi(n_\eps)-y^{1/2}\xi\big)dx \\
	&\phantom{xx}{}+\int_\Omega |\xi|^2\big(y-y_\eps\big)dx
	\to 0, \nonumber
\end{align}
applying the dominated convergence theorem to the last integral. 
Furthermore, using the test function $\phi=y(\Psi(n_\eps)-\Psi(n))$ in \eqref{auxeq},
\begin{align*}
  \bigg| & \int_\Omega y_\eps y^{1/2}\na \Psi(n_\eps)\cdot
	\big(y^{1/2}\na \Psi(n_\eps)-\xi\big)dx\bigg| \\
	&= \bigg|\int_\Omega\Big(y_\eps y|\na\Psi(n_\eps)|^2 - y_\eps\na\Psi(n_\eps)
	\cdot\big(\na(y\Psi(n))-\Psi(n)\na y\big)\Big)dx\bigg|\\
	&= \bigg|\int_\Omega y_\eps\na \Psi(n_\eps)\cdot\na(y(\Psi(n_\eps)-\Psi(n))) 
	- \int_\Omega y_\eps^{1/2}\na\Psi(n_\eps) 
	\cdot\na y\big(y_\eps^{1/2}(\Psi(n_\eps)-\Psi(n))\big)dx \bigg| \\
	&= \bigg|\frac{1}{\dt}\int_\Omega(\overline{n}-n_\eps)y(\Psi(n_\eps)-\Psi(n))dx
	- \int_\Omega y_\eps^{1/2}\na\Psi(n_\eps)\cdot\na y\big(y_\eps^{1/2}
	(\Psi(n_\eps)-\Psi(n))\big)dx \bigg| \\
	&\le \frac{1}{\dt}\|n_\eps-\overline{n}\|_{L^2(\Omega)}
	\|y(\Psi(n_\eps)-\Psi(n))\|_{L^2(\Omega)}	\\
	&\phantom{xx}{}+ \|y_\eps^{1/2}\na\Psi(n_\eps)\|_{L^2(\Omega)}
	\bigg(\int_\Omega y_\eps(\Psi(n_\eps)-\Psi(n))^2|\na y|^2 dx\bigg).
\end{align*}

By \eqref{auxnn}, we have $\|y(\Psi(n_\eps)-\Psi(n))\|_{L^2(\Omega)}\to 0$ 
and by \eqref{auxyn}, $y_\eps(\Psi(n_\eps)-\Psi(n))^2\to 0$ in $\Omega$ for 
a subsequence. Then, by dominated convergence, 
$\int_\Omega y_\eps(\Psi(n_\eps)-\Psi(n))^2|\na y|^2 dx\to 0$.
We have proved that
\begin{equation}\label{auxc2}
  \int_\Omega y_\eps y^{1/2}\na \Psi(n_\eps)\cdot
	\big(y^{1/2}\na \Psi(n_\eps) - \xi\big)dx \to 0.
\end{equation}
Subtracting \eqref{auxc1} from \eqref{auxc2}, we conclude that
$$
  \int_\Omega y_\eps|y^{1/2}\na \Psi(n_\eps)-\xi|^2 dx \to 0.
$$
Taking into account this convergence and \eqref{auxyy}, it follows again by 
the dominated convergence theorem that
\begin{align*}
  \int_\Omega  yy_\eps|\na \Psi(n_\eps)|^2 dx
	&= \int_\Omega y_\eps|y^{1/2}\na \Psi(n_\eps)-\xi|^2dx \\
	&\phantom{xx}{}+2\int_\Omega y_\eps\big(y^{1/2}\na \Psi(n_\eps)-\xi\big)\cdot 
	\xi dx + \int_\Omega y_\eps|\xi|^2dx \to \int_\Omega y|\xi|^2 dx.
\end{align*}
This shows the first part of the proposition. 

{\em Step 6.}
It remains to show that the limit $n$ solves \eqref{auxeq2}. 
Let $\phi\in W^{1,\infty}(\Omega)$. Since (a subsequence of) $(n_\eps)$
converges weakly* to $n$ in $L^\infty(\Omega)$, we have
$$
  \frac{1}{\dt}\int_\Omega(n_\eps-\overline{n})\phi dx\to
	\frac{1}{\dt}\int_\Omega(n-\overline{n})\phi dx.
$$
Furthermore,
\begin{align*}
  \int_\Omega y_\eps\na\Psi(n_\eps)\cdot\na\phi dx
	&= \int_\Omega y^{1/2}y_\eps^{1/2}\na\Psi(n_\eps)\cdot\na\phi dx \\
	&\phantom{xx}{}
	+ \int_\Omega (y^{1/2}-y_\eps^{1/2})y_\eps^{1/2}\na\Psi(n_\eps)\cdot\na\phi dx.
\end{align*}
By Step 5, the first integral converges to 
$$
  \int_\Omega\big(\na(y\Psi(n))-\Psi(n)\na y\big)\cdot\na\phi dx,
$$
while the second integral converges to zero since
\begin{align*}
  \bigg|\int_\Omega & (y^{1/2}-y_\eps^{1/2})y_\eps^{1/2}\na\Psi(n_\eps)\cdot\na\phi dx
	\bigg| \\
	&\le \|y_\eps^{1/2}\na\Psi(n_\eps)\|_{L^2(\Omega)}
	\|y^{1/2}-y_\eps^{1/2}\|_{L^2(\Omega)}\|\na\phi\|_{L^\infty(\Omega)}\to 0.
\end{align*}
We conclude that \eqref{auxeq2} holds in the weak sense for test functions
in $W^{1,\infty}(\Omega)$ but a density argument shows that it is sufficient
to take test functions in $H^1(\Omega)$. This finishes the proof.
\end{proof}


\section{Existence of solutions to the high-temperature model}\label{sec.ex}

We prove the existence of weak solutions to \eqref{2.eq} in $\T^d$. 
We recall the definition of the total (``reverted'') energy
\begin{equation}\label{ex.wtot}
  W_{\rm tot}^k = W^k - \frac{U}{2}(n^k)^2
\end{equation}
and introduce the total variance
\begin{equation}\label{var}
  V^k := \int_{\T^d}\bigg((W^k)^2 - \int_{\T^d}W^kdz\bigg)^2dx
	+ U\int_{\T^d}W^{k-1}dx\int_{\T^d}\bigg(n^k - \int_{\T^d}n^kdz\bigg)^2 dx.
\end{equation}
The main result is as follows.

\begin{theorem}[Existence of weak solutions]\label{thm.ex}
Let $\dt>0$, $U>0$, $\eta\in(0,1]$, $0<\delta<1/(1+\eta)$ and let 
$$
  n^{k-1},\ W^{k-1}\in L^\infty(\T^d), \quad
  \delta \le n^{k-1}\le\frac{1-\delta}{\eta}, \quad W^{k-1}\ge 0\mbox{ in }\T^d.
$$
Then there exists a weak solution $(n^k,W^k)$ to \eqref{1.appn}-\eqref{1.appw} 
in the following sense: It holds $\delta\le n^k\le\|n^{k-1}\|_{L^\infty(\T^d)}
\le (1-\delta)/\eta$, $0\le W^k\le\|W^{k-1}\|_{L^\infty(\T^d)}$ in $\T^d$,
$W^kn^k$, $W^k\in H^1(\T^d)$, as well as
\begin{align}
  \frac{1}{\dt}\int_{\T^d}(n^k-n^{k-1})\phi_0 dx
  &= -\int_\Omega\frac{\na(W^kn^k)-n^k\na W^k}{g(n^k)}
	\cdot\na\phi_0 dx, \label{1.ntau} \\
	\frac{1}{\dt}\int_{\T^d}(W^k-W^{k-1})W^k\phi_1 dx
  &= -\frac{2d-1}{2d}\int_{\T^d}\frac{\na W^k\cdot\na(W^k\phi_1)}{g(n^k)}dx
	\label{1.wtau} \\
  &\phantom{xx}{}
	- U\int_{\T^d}\frac{|\na (W^kn^k)-n^k\na W^k|^2}{g(n^k)} \phi_1 dx \nonumber
\end{align}
for all $\phi_0\in H^1(\T^d)$ and $\phi_1\in H^1(\T^d)\cap L^\infty(\T^d)$,
where $g(n^k)=n^k(1-\eta n^k)$. 
For this solution, the following monotonicity properties hold:
\begin{equation}\label{ex.mono}
  \int_{\T^d}W_{\rm tot}^{k}dx \ge \int_{\T^d}W_{\rm tot}^{k-1} dx, \quad
	V^k + \dt\frac{2d-1}{d}\int_{\T^d}\frac{|\na W^k|^2}{g(n^k)}dx \le V^{k-1},
\end{equation}
where $W_{\rm tot}^k$ and $V^k$ are defined in \eqref{ex.wtot} and
\eqref{var}, respectively. Moreover, if
\begin{equation}\label{ex.assump}
  \frac{U}{2} \int_{\T^d}\bigg(n^{k-1}-\int_{\T^d}n^{k-1}dz\bigg)^2dx
	< \int_{\T^d}W^{k-1}dx
\end{equation}
holds then $W^k\not\equiv 0$.
\end{theorem}

\begin{remark}[Comments]\label{rem.comm}\rm
1.\ The existence result holds true for more general functions $g(n)$ under the
assumption that $g(n)$ is strictly positive for $\delta\le n\le(1-\delta)/\eta$.

2.\ One may interpret $W^k$ as a ``renormalized'' solution  since we need test
functions of the form $W^k\phi_1$ in order to avoid vacuum sets $W^k=0$.
Such an idea has been used, for instance, for the compressible quantum 
Navier-Stokes equations to avoid vacuum sets in the particle density \cite{Jue10}.
Test functions of the type $W^k\phi$
allow for the trivial solution $n^k=n^{k-1}$ and $W^k=0$
but assumption \eqref{ex.assump} excludes this situation. It means that no
constant steady state with $W^k=0$ exists if the 
variance of $n^{k-1}$ is small compared to the energy $\int_{\T^d}W^{k-1}dx$.

3.\ The second inequality in \eqref{ex.mono} involves $W^{k-2}$ which 
makes sense when the equations are solved iteratively, starting from $k=1$.	
Also \eqref{ex.assump} can be iterated. Indeed, if \eqref{ex.assump}
holds for $(n^{k-1},W^{k-1})$, the monotonicity property \eqref{ex.mono} and
mass conservation $\int_{\T^d}n^k dx=\int_{\T^d}n^{k-1}dx$ imply that
\begin{align*}
  \frac{U}{2} & \int_{\T^d}\bigg(n^{k} - \int_{\T^d}n^{k}dz\bigg)^2
	= \frac{U}{2}\int_{\T^d}(n^k)^2 dx - \frac{U}{2}\bigg(\int_{\T^d}n^k dx\bigg)^2 \\
	&\le \int_{\T^d}(W^k-W^{k-1})dx + \frac{U}{2}\int_{\T^d}(n^{k-1})^2 dx
	- \frac{U}{2}\bigg(\int_{\T^d}n^{k-1} dx\bigg)^2 
	< \int_{\T^d} W^k dx.
\end{align*}

4.\ We are not able to perform the limit $\dt\to 0$. The reason is that
we cannot perform the limit in the quadratic gradient term 
$|\na (W^kn^k)-n^k\na W^k|^2$, since we cannot prove the strong convergence of 
$\na (W^kn^k)-n^k\na W^k$.
Proposition \ref{prop.conv} provides such a result for the time-discrete
elliptic case. The key step is to show that
\begin{align*}
  \int_\Omega \frac{\na(W^kn^k)-n^k\na W^k}{g(n^k)}
  &\cdot\na(Wn-W^kn^k) dx \\
  &= \frac{1}{\dt}\big\langle n^k-n^{k-1},W^kn^k-Wn\big\rangle \to 0,
\end{align*}
where $n$, $W$ are the (weak) limits of $(n^k)$, $(W^k)$, respectively, 
and $\langle\cdot,\cdot\rangle$
is the dual product between $H^1(\T^d)'$ and $H^1(\T^d)$. It is
possible to show that $\dt^{-1}(n^k-n^{k-1})$ is bounded
in $H^1(\T^d)'$, but the limit $W^kn^k-Wn\to 0$
strongly in $H^1(\T^d)$ (more precisely: the limit of the piecewise constant in
time construction of $W^kn^k$ in $L^2(0,T;H^1(\T^d))$) cannot be expected.
\qed
\end{remark}

In the one-dimensional case and under the smallness condition \eqref{assump2}
below, we can show that $W^k$ is positive, which allows us to define
the weak solution to \eqref{1.appn}-\eqref{1.appw} in the standard sense
(with test functions $\phi_1$ instead of $W^k\phi_1$). We set
$$
  \overline{W^{k-1}} = \int_{\T}W^{k-1}dx, \quad
	\overline{n^{k-1}} = \int_{\T}n^{k-1}dx.
$$
\begin{theorem}[One-dimensional case]\label{thm.ex2}
Let the assumptions of Theorem \ref{thm.ex} hold, let $d=1$,
$G=\max_{\delta\le s\le \|n^{k-1}\|_{L^\infty(\T)}}g(s)$, 
and let $(n^k,W^k)$ for $k\ge 0$ be the solution
given by Theorem \ref{thm.ex}. We assume that
\begin{equation}\label{assump2}
  \frac{G}{\dt}\big\|W^{k-1}-\overline {W^{k-1}}\big\|_{L^2(\T)}^2 
	+ U\bigg(\frac{G\overline {W^{k-2}}}{\dt}+\frac12\bigg)
	\big\|n^{k-1}-\overline {n^{k-1}}\big\|_{L^2(\T)}^2 < \overline {W^{k-1}},
\end{equation}
Then $W^k$ is strictly positive, $n^k\in H^1(\T)$, and 
\eqref{1.appn}-\eqref{1.appw} hold in the sense of $H^1(\T)'$. 
\end{theorem}

We proceed to the proof of Theorems \ref{thm.ex} and \ref{thm.ex2}.
In this section, $\eps$ denotes a positive parameter and not the band energy.
Since we are not concerned with the kinetic equations, no notational confusion
will occur. Let $\dt>0$, $\alpha$, $\gamma$, $\delta$, $\eps>0$
satisfying $\gamma<1$ and $\delta<1/(1+\eta)$.
Define the truncations
$$
  [n]_\delta = \max\big\{\delta,\min\{(1-\delta)/\eta,n\}\big\}, \quad
	[W]_\gamma = \max\big\{0,\min\{1/\gamma,W\}\big\},
$$
and $g_\delta(s)=[s]_\delta(1-\eta[s]_\delta)$ for $s\in\R$. Then
$g_\delta$ is continuous and strictly positive.
Given $W^{k-1}$, $n^{k-1}\in L^\infty(\T^d)$ satisfying 
$\delta\le n^{k-1}\le (1-\delta)/\eta$, we solve the regularized and
truncated nonlinear problem in $\T^d$
\begin{align}
  \frac{1}{\dt}(n^k-n^{k-1}) 
	&= \diver\bigg(\frac{[W^k]_\gamma+\eps}{g_\delta(n^k)}
	\na n^k\bigg), \label{ex.appn} \\
	\frac{1}{\dt}(W^k-W^{k-1}) 
	&= \frac{2d-1}{2d}\diver\bigg(\frac{\na W^k}{g_\delta(n^k)}\bigg)
	- U\frac{[W^k]_\gamma}{g_\delta(n^k)}\frac{|\na n^k|^2}{1+\alpha|\na n^k|^2}.
  \label{ex.appw}
\end{align}
	
\begin{remark}\rm
Let us explain the approximation \eqref{ex.appn}-\eqref{ex.appw}. 
The truncation of $W^k$ with parameter $\gamma$ ensures that the coefficients 
are bounded, while the truncation $g_\delta(n^k)$ with parameter $\delta>0$ 
guarantees that the denominator is always positive.
The regularization parameter $\eps$ gives strict ellipticity for \eqref{ex.appn},
since generally the first term on the right-hand side of \eqref{ex.appn} 
without $\eps$ is degenerate. 
Finally, the approximation of the quadratic gradient term with parameter $\alpha$ 
avoids regularity issues since it holds $|\na n^k|^2\in L^1(\T^d)$ only.
\qed
\end{remark}

\subsection{Solution of an approximated problem}
	
First, we prove the existence of solutions to \eqref{ex.appn}-\eqref{ex.appw}.
	
\begin{lemma}[Existence for the approximated problem]
There exists a weak solution $(n^k,$ $W^k)\in H^1(\T^d)^2$ to 
\eqref{ex.appn}-\eqref{ex.appw}.
\end{lemma}

\begin{proof}
We define the fixed-point operator $S:L^2(\T^d)^2\times[0,1]\to L^2(\T^d)^2$
by $S(n^*,W^*;\theta)=(n,W)$, where $(n,W)\in H^1(\T^d)^2$ 
is the unique solution to the linear problem
\begin{equation}\label{ex.LM}
  a_0(n,\phi_0) = F_0(\phi_0), \quad a_1(W,\phi_1)=F_1(\phi_1)
	\quad\mbox{for all }\phi_0,\,\phi_1\in H^1(\T^d),
\end{equation}
where 
\begin{align*}
  a_0(n,\phi_0) &= \int_{\T^d}\frac{[W^*]_\gamma+\eps}{g_\delta(n^*)}s
	\na n\cdot\na\phi_0 dx + \frac{1}{\dt}\int_{\T^d}n\phi_0 dx, \\
	F_0(\phi_0) &= \frac{\theta}{\dt}\int_{\T^d} n^{k-1}\phi_0 dx, \\
	a_1(W,\phi_1) &= \frac{2d-1}{2d}\int_{\T^d}
	\frac{\na W\cdot\na\phi_1}{g_\delta(n^*)}dx 
	+ \frac{1}{\dt}\int_{\T^d}W\phi_1 dx, \\
	F_1(\phi_1) &= \frac{\theta}{\dt}\int_{\T^d}W^{k-1}\phi_1 dx
	- \theta U\int_{\T^d}\frac{[W^*]_\gamma}{g_\delta(n^*)}
	\frac{|\na n|^2}{1+\alpha|\na n|^2}\phi_1 dx.
\end{align*}
The approximation and truncation ensure that these forms are bounded on
$H^1(\T^d)$. The bilinear forms $a_0$ and $a_1$ are coercive. By the
Lax-Milgram lemma, there exists a unique solution $(n,W)\in H^1(\T^d)^2$
to \eqref{ex.LM}. Thus, the fixed-point operator is well defined (and
has compact range). Furthermore, $S(n^*,W^*;0)=0$. 
Standard arguments show that $S$ is continuous. Let $(n,W)$ be a fixed point
of $S(\cdot,\cdot;\theta)$, i.e., $(n,W)$ solves \eqref{ex.appn}-\eqref{ex.appw}
with $(n^k,W^k)$ replaced by $(n,W)$. With the test functions $\phi_0=n$ and
$\phi_1=W$ and the inequality 
$$
  \bigg(\frac{1}{\dt}n - \frac{\theta}{\dt}n^{k-1}\bigg)n
	\ge \frac{1}{2\dt}\big(n^2 - (n^{k-1})^2\big),
$$
we find that
\begin{align*}
  \frac{1}{2\dt}\int_{\T^d}n^2 dx 
	+ \int_{\T^d}\frac{[W]_\gamma+\eps}{g_\delta(n)}|\na n|^2 dx
	&\le \frac{1}{2\dt}\int_{\T^d}(n^{k-1})^2 dx, \\
	\frac{1}{2\dt}\int_{\T^d}W^2 dx
	+ \frac{2d-1}{2d}\int_{\T^d}\frac{|\na W|^2}{g_\delta(n)}dx
	&\le \frac{1}{2\dt}\int_{\T^d}(W^{k-1})^2 dx \\
	&\phantom{xx}{}
	- \theta U\int_{\T^d}\frac{[W]_\gamma W}{g_\delta(n)}
	\frac{|\na n|^2}{1+\alpha|\na n|^2}dx.
\end{align*}
The last integral is nonnegative since $[W]_\gamma W\ge 0$. Therefore,
$$
  \|n\|_{H^1(\T^d)} \le C(\eps), \quad
	\|W\|_{H^1(\T^d)} \le C(\delta),
$$
where $C(\eps)$ and $C(\delta)$ are positive constants independent
of $(n,W)$. This provides the necessary uniform bound for all fixed points
of $S(\cdot,\cdot;\theta)$. 
We can apply the Leray-Schauder fixed-point theorem to infer the
existence of a fixed point for $S(\cdot,\cdot;1)$, i.e.\ of a weak
solution to \eqref{ex.appn}-\eqref{ex.appw}.
\end{proof}

\subsection{Removing the truncation}

The following maximum principle holds.

\begin{lemma}[Maximum principle]
Let $(n^k,W^k)$ be a weak solution to \eqref{ex.appn}-\eqref{ex.appw}. Then
$$
  \delta\le n^k\le \|n^{k-1}\|_{L^\infty(\T^d)}\le \frac{1-\delta}{\eta}, 
	\quad 0\le W^k\le \frac{1}{\gamma}\quad\mbox{in }\T^d,
$$
where $\gamma\le 1/\|W^{k-1}\|_{L^\infty(\T^d)}$.
\end{lemma}

\begin{proof}
We choose $(n^k-\delta)_-=\min\{0,n^k-\delta\}$ as a test function in \eqref{ex.appn}:
\begin{align*}
  \frac{1}{\dt}\int_{\T^d} &\big((n^k-\delta)-(n^{k-1}-\delta)\big)(n^k-\delta)_-dx \\
	&{}+ \int_{\T^d}\frac{[W^k]_\gamma+\eps}{g_\delta(n^k)}
	\na n^k\cdot\na(n^k-\delta)_- dx = 0.
\end{align*}
Since $-(n^{k-1}-\delta)(n^k-\delta)_-\ge 0$, this gives
$$
  \frac{1}{\dt}\int_{\T^d}(n^k-\delta)_-^2 dx
	\le -\int_{\T^d}\frac{[W^k]_\gamma+\eps}{g_\delta(n^k)}|\na(n^k-\delta)_-|^2 dx
	\le 0,
$$
and hence, $n^k\ge\delta$ in $\T^d$. In a similar way, the test function
$(n^k-N)_+=\max\{0,n^k-N\}$ with $N:=\|n^{k-1}\|_{L^\infty(\T^d)}$ leads to $n^k-N\le 0$
in $\T^d$. Next, we use $W^k_-\le 0$ as a test function in \eqref{ex.appw}:
\begin{align*}
  \frac{1}{\dt}\int_{\T^d} & (W_-^k)^2 dx 
	+ \frac{2d-1}{2d}\int_{\T^d}\frac{|\na W_-|^2}{g_\delta(n^k)}dx \\
  &= \frac{1}{\dt}\int_{\T^d}W^{k-1}W_-^k dx 
	- U\int_{\T^d}\frac{[W^k]_\gamma W_-^k}{g_\delta(n^k)}
	\frac{|\na n^k|^2}{1+\alpha|\na n^k|^2}dx \le 0.
\end{align*}
We deduce that $W^k\ge 0$. The proof of $W^k\le \|W^{k-1}\|_{L^\infty(\T^d)}
\le 1/\gamma$
is similar, using the test function $(W^k-\|W^{k-1}\|_{L^\infty(\T^d)})_+$.
\end{proof}

We have shown that $(n^k,W^k)$ solves
\begin{align}
  \frac{1}{\dt}(n^k-n^{k-1}) &= \diver\bigg(\frac{W^k+\eps}{g(n^k)}
	\na n^k\bigg), \label{ex.aln} \\
	\frac{1}{\dt}(W^k-W^{k-1}) 
	&= \frac{2d-1}{2d}\diver\bigg(\frac{\na W^k}{g(n^k)}\bigg)
	- U\frac{W^k}{g(n^k)}\frac{|\na n^k|^2}{1+\alpha|\na n^k|^2},
  \label{ex.alw}
\end{align}
where $g(n)=n(1-\eta n)$.

\subsection{The limit $\alpha\to 0$}

Let $(n_\alpha^k,W_\alpha^k)$ be a weak solution to \eqref{ex.aln}-\eqref{ex.alw}.
We use the test function $n^k_\alpha$ in \eqref{ex.aln},
\begin{equation}\label{ex.auxn}
  \frac{1}{2\dt}\int_{\T^d}(n^k_\alpha)^2 dx 
	+ \int_{\T^d}\frac{W^k_\alpha+\eps}{g(n^k_\alpha)}|\na n^k_\alpha|^2 dx
	\le \frac{1}{2\dt}\int_{\T^d}(n^{k-1})^2 dx,
\end{equation}
and the test function $W^k_\alpha$ in \eqref{ex.alw},
\begin{equation}\label{ex.auxw}
  \frac{1}{2\dt}\int_{\T^d}(W^k_\alpha)^2 dx
	+ \frac{2d-1}{2d}\int_{\T^d}\frac{|\na W_\alpha^k|^2}{g(n^k_\alpha)}
	\le \frac{1}{2\dt}\int_{\T^d}(W^{k-1})^2 dx,
\end{equation}
which provides immediately uniform $H^1(\T^d)$ estimates since $g(n^k_\alpha)\ge
C(\delta)>0$:
$$
  \|n^k_\alpha\|_{H^1(\T^d)} \le C(\delta,\eps,\dt), \quad 
	\|W^k_\alpha\|_{H^1(\T^d)} \le C(\delta,\dt),
$$
where the constants are independent of $\alpha$. By compactness, 
this implies the existence of a subsequence which is not relabeled such that,
as $\alpha\to 0$, 
\begin{align*}
  n^k_\alpha\to n^k,\ W^k_\alpha\to W^k &\quad\mbox{strongly in }L^2(\T^d), \\
	n^k_\alpha\rightharpoonup n^k,\ W^k_\alpha\rightharpoonup W^k
	&\quad\mbox{weakly in }H^1(\T^d).
\end{align*}
This shows that, maybe for a subsequence, $W^k_\alpha/g(n^k_\alpha)\to
W^k/g(n^k)$ and $1/g(n^k_\alpha)\to 1/g(n^k)$ a.e.\ in $\T^d$, and by
dominated convergence, strongly in $L^2(\T^d)$.

We claim that $n^k_\alpha\to n^k$ strongly in $H^1(\T^d)$.
Let $y_\alpha:=(W^k_\alpha+\eps)/g(n^k_\alpha)$. Then 
$y_\eps\ge\eps/\sup_{s\in[\delta,N]}g(s)>0$, where $N=\|n^{k-1}\|_{L^\infty(\T^d)}$, 
and $y_\alpha\to y:=W^k/g(n^k)\ge 0$ strongly in $L^2(\Omega)$. Thus,
$y_\alpha^{1/2}\na n^k_\alpha\rightharpoonup y^{1/2}\na n^k$
weakly in $L^2(\Omega)$, and it follows that
$$
  \int_{\T^d}y_\alpha \na n^k\cdot\na(n^k_\alpha-n^k)dx \to 0.
$$
Taking $n^k_\alpha-n^k$ as a test function in \eqref{ex.aln}, we obtain
$$
  \int_{\T^d}y_\alpha\na n^k_\alpha\cdot\na(n^k_\alpha-n^k)dx
	= -\frac{1}{\dt}\int_\Omega(n^k_\alpha-n^{k-1})(n^k_\alpha-n^k)dx \to 0.
$$
Subtraction of these integrals leads to
$$
  \int_{\T^d}y_\alpha|\na(n^k_\alpha-n^k)|^2 dx \to 0.
$$
Since $y_\alpha\ge\eps/\sup_{s\in[\delta,N]}g(s)>0$, this proves the claim. 
In particular, $1/(1+\alpha|\na n^k_\alpha|^2)\to 1$ in $L^2(\T^d)$. From this, 
we can directly deduce that
$$
  \frac{|\na n^k_\alpha|^2}{1+\alpha|\na n^k_\alpha|^2}
	\to |\na n^k|^2 \quad \mbox{in }L^1(\T^d).
$$
The above convergence results are sufficient to pass to the limit
$\alpha\to 0$ in \eqref{ex.aln}-\eqref{ex.alw}, showing that $(n^k,W^k)$
solves
\begin{align}
  \frac{1}{\dt}(n^k-n^{k-1}) &= \diver\bigg(\frac{W^k+\eps}{g(n^k)}
	\na n^k\bigg), \label{ex.epsn} \\
	\frac{1}{\dt}(W^k-W^{k-1}) 
	&= \frac{2d-1}{2d}\diver\bigg(\frac{\na W^k}{g(n^k)}\bigg)
	- U\frac{W^k}{g(n^k)}|\na n^k|^2,
  \label{ex.epsw}
\end{align}

\subsection{The limit $\eps\to 0$}

This limit is the delicate part of the proof. We first state a lemma 
concerning weak and strong convergence.

\begin{lemma}\label{lem.weak.strong}
Let $(f_n)$ be a weakly and $(g_n)$ be a strongly converging sequence in 
$L^2(\Omega)$ which have the same limit. If $|f_n(x)|\leq |g_n(x)|$ for all 
$n\in \mathbb N$ and a.e.\ $x\in\T^d$, then $(f_n)$ converges strongly in
$L^2(\Omega)$.
\end{lemma}

\begin{proof}
Let $f$ denote the weak limit of $(f_n)$ and $(g_n)$. Due to the weak lower 
semi-continuity of the norm, 
\begin{align*}
  \int_{\T^d}|f(x)|^2dx &\leq \liminf_{n\to\infty}
  \int_{\T^d}|f_n(x)|^2dx \\&\leq \limsup_{n\to\infty}
  \int_{\T^d}|f_n(x)|^2dx \leq \lim_{n\to\infty}
  \int_{\T^d}|g_n(x)|^2dx = 
  \int_{\T^d}|f(x)|^2dx.
	\end{align*}
Thus, the limes inferior and superior coincide and $\|f_n\|_{L^2(\Omega)}\to
\|f\|_{L^2(\Omega)}$. Together with the weak convergence of $(f_n)$, 
we deduce the strong convergence.
\end{proof}
	
Let $(n^k_\eps,W^k_\eps)$ be a weak solution to \eqref{ex.epsn}-\eqref{ex.epsw}.
Inequalities \eqref{ex.auxn} and \eqref{ex.auxw} show the following bounds
uniform in $\eps$:
\begin{align}
  \|n^k_\eps\|_{L^\infty(\T^d)} 
	+ \|(W^k_\eps+\eps)^{1/2}\na n^k_\eps\|_{L^2(\T^d)}
  &\le C(\delta,\dt), \label{ex.eps1} \\
	\|W^k_\eps\|_{L^\infty(\Omega)} + \|W^k_\eps\|_{H^1(\T^d)}
	&\le C(\delta,\dt). \label{ex.eps2}
\end{align}
By compactness, there exists a subsequence (not relabeled) such that, as $\eps\to 0$,
\begin{align}
 	& n_\eps^k\rightharpoonup^* n^k \quad\mbox{weakly* in }L^\infty(\T^d), 
  \label{ex.neps} \\
  & W^k_\eps\to W^k\quad\mbox{strongly in }L^2(\T^d), \quad
	W^k_\eps\rightharpoonup W^k\quad\mbox{weakly in }H^1(\T^d). \label{ex.weps}
\end{align}

Again, we need strong convergence for $\na n^k_\eps$. Since equation
\eqref{ex.epsn} is degenerate, we obtain a weaker result. 
For this, let $\Psi\in C^2(\R)$ be strictly monotonically increasing and satisfy
$\Psi'(t)=1/g(t)$ for $\delta\leq t\leq (1-\delta)/\eta$. Thus, we can apply 
Proposition \ref{prop.conv} for $y_\eps=W^k_\eps+\eps$ and $y=W^k$
to conclude that, up to a subsequence,
\begin{align}\label{ex.prop}
  & (W^k)^{1/2}(W^k_\eps+\eps)^{1/2}\na \Psi(n_\eps^k)
	\to \na(W^k\Psi(n^k))-\Psi(n^k)\na W^k, \\
	& \mathrm{1}_{\{W^k>0\}}(n_\eps^k-n^k) \to 0 \quad\mbox{strongly in }
	L^2(\Omega). \nonumber
\end{align}
The latter convergence implies that $n^k_\eps-n^k\to 0$ a.e.\
in $\{W^k>0\}$ and, by dominated convergence,
\begin{equation}\label{ex.pos}
\frac{\mathrm{1}_{\{W^k>0\}}}{g(n^k_\eps)} 
\to \frac{\mathrm{1}_{\{W^k>0\}}}{g(n^k)}\quad\mbox{strongly in }L^2(\T^d).
\end{equation}

Now, let $y_\eps=(W^k_\eps+\eps)/g(n^k_\eps)$ and $y=W^k/g(n^k)$. We know that
$y_\eps$, $y\in L^\infty(\T^d)\cap H^1(\T^d)$ and $ y_\eps\to y$ strongly in 
$L^2(\T^d)$. Thus, we can again apply Proposition  \ref{prop.conv} with 
$\Psi=\mathrm{Id}$ and infer that $W^kn^k\in H^1(\T^d)$ as well as 
	\begin{equation*}
	  \frac{1}{\dt}\int_{\T^d}(n^k-n^{k-1})\phi_0 dx
	  +\int_\Omega\frac{\na(W^kn^k)-n^k\na W^k}{g(n^k)}
	  \cdot\na\phi_0 dx=0
	\end{equation*}
for all $\phi_0\in H^1(\T^d)$.

Let $\phi_1$ be a smooth test function. 
We use the test function $W^k\phi_1$ in the weak formulation of \eqref{ex.epsw}:
\begin{align}\label{approx.eq.Wk.eps}
	0 &= \frac{1}{\dt}\int_{\T^d}(W^k_\eps-W^{k-1})W^k\phi_1 dx
	+ \frac{2d-1}{2d}\int_{\T^d}
	\frac{\na W^k_\eps\cdot\na(W^k\phi_1)}{g(n^k_\eps)}dx \\
	&\phantom{xx}{}+ U\int_{\T^d}W^k\frac{W^k_\eps}{g(n^k_\eps)}
	|\na n^k_\eps|^2\phi_1 dx =: I^1_\eps+I^2_\eps+I^3_\eps. \nonumber
\end{align}
We pass to the limit $\eps\to 0$ term by term. By \eqref{ex.weps}, 
$$
  I^1_\eps	\to \frac{1}{\dt}\int_{\T^d}(W^k-W^{k-1})W^k\phi_1 dx.
$$
For the integral $I^2_\eps$, we use the strong convergence \eqref{ex.pos} and
the weak convergence of $(\na W^k_\eps)$ in $L^2(\T^d)$ to infer that
\begin{align*}
  I^2_\eps &= \int_{\T^d}\frac{\mathrm{1}_{\{W^k>0\}}}{g(n^k_\eps)}
	W^k\na W^k_\eps\cdot\na\phi_1 dx
	+ \int_{\T^d}\frac{\mathrm{1}_{\{W^k>0\}}}{g(n^k_\eps)}\na W^k_\eps\cdot\na W^k
	\phi_1 dx \\
	&\to \int_{\T^d}\frac{\mathrm{1}_{\{W^k>0\}}}{g(n^k)}W^k\na W^k\cdot\na\phi_1 dx
	+ \int_{\T^d}\frac{\mathrm{1}_{\{W^k>0\}}}{g(n^k)}|\na W^k|^2\phi_1 dx \\
	&= \int_{\T^d}\frac{\na W^k\cdot\na(W^k\phi_1)}{g(n^k)}dx.
\end{align*}

The remaining integral $I^3_\eps$ requires some work.
As a preparation, using Proposition \ref{prop.conv},
we infer similarly to \eqref{ex.prop} that
$$
	h_\eps :=(W^k)^{1/2}\left(\frac{W^k_\eps+\eps}{g(n^k_\eps)}\right)^{1/2}\na n_\eps^k
	\to \frac{\na(W^kn^k)-n^k\na W^k}{g(n^k)^{1/2}}=:h
$$
strongly in $L^2(\T^d)$. Let
$$
	\xi_\eps :=\bigg(\frac{W^k_\eps}{g(n^k_\eps)}\bigg)^{1/2}\na n_\eps^k.
$$
Then $(\xi_\eps)$ is bounded in $L^2(\Omega)$ and admits a weakly convergent 
subsequence, i.e.\ $\xi_\eps\rightharpoonup\xi$ for some $\xi\in L^2(\T^d)$. 
Similarly as in the proof of Proposition \ref{prop.conv}, i.e.\ with
an argument as in \eqref{auxabschaetzung_gut_fuer_spaeter}, we can find that, 
up to a subsequence, $(W^k_\eps+\eps)^{1/2}\xi_\eps\rightharpoonup h$ weakly in 
$L^2(\T^d)$ implying $(W^k)^{1/2}\xi=h$. In particular,
$$
	f_\eps:=(W^k)^{1/2}\xi_\eps\rightharpoonup(W^k)^{1/2}\xi = h
	\quad\mbox{ weakly in }L^2(\T^d).
$$
Since $|f_\eps(x)|\leq |h_\eps(x)|$ for a.e.\ $x\in\T^d$ and all $\eps>0$, we 
can apply Lemma \ref{lem.weak.strong} and obtain that, up to a subsequence, 
$f_\eps$ converges strongly in $L^2(\Omega)$. Thus,
$$
	I^3_\eps = U\int_{\T^d}f_\eps^2\phi_1 dx
	\to U\int_{\T^d}\frac{|\na(W^kn^k)-n^k\na W^k|^2}{g(n^k)}\phi_1 dx. 
$$
Hence, passing to the limit $\eps\to 0$ in \eqref{approx.eq.Wk.eps},
we infer that $(n^k,W^k)$ solves \eqref{1.appn}-\eqref{1.appw}.

\subsection{Energy estimate}

We claim that the total energy $\int_{\T^d}W_{\rm tot}^k dx$ is nondecreasing in $k$.
Let $(n^k_\eps,W^k_\eps)$ be a weak solution to \eqref{ex.epsn}-\eqref{ex.epsw}.
Then
\begin{align*}
  \dt\int_{\T^d} & \bigg(W^k_\eps-W^{k-1} 
	- \frac{U}{2}\big((n_\eps^k)^2-(n^{k-1})^2\big)\bigg)dx \\
  &\ge \dt\int_{\T^d}\big(W^k_\eps-W^{k-1} - U(n_\eps^k-n^{k-1})n_\eps^k\big)dx.
\end{align*}
Taking the test functions $\phi_1=U$ in \eqref{ex.epsw} and $\phi_0=n_\eps^k$
in \eqref{ex.epsn} and subtracting both equations, the above integral becomes
\begin{align*}
  \dt\int_{\T^d} & \bigg(W^k_\eps-W^{k-1} 
	- \frac{U}{2}\big((n_\eps^k)^2-(n^{k-1})^2\big)\bigg)dx \\
  &\ge -U\int_{\T^d}\frac{W_\eps^k}{g(n_\eps^k)}|\na n_\eps^k|^2 dx
	+ U\int_{\T^d}\frac{W_\eps^k+\eps}{g(n_\eps^k)}|\na n_\eps^k|^2 dx \ge 0.
\end{align*}
Thus, with the lower semi-continuity of the norm, we have
\begin{align*}
  \int_{\T^d}\bigg(W^{k-1}-\frac{U}{2}(n^{k-1})^2\bigg)dx
  &\le \liminf_{\eps\to 0}
	\int_{\T^d}\bigg(W_\eps^k-\frac{U}{2}(n^k_\eps)^2\bigg) dx \\
	&\le \int_{\T^d}\bigg(W^{k}-\frac{U}{2}(n^{k})^2\bigg)dx.
\end{align*}
In view of mass conservation $\int_{\T^d}n^kdx=\int_{\T^d}n^{k-1}dx$, 
it follows by Jensen's inequality that
\begin{align}
  \int_{\T^d}W^k dx 
  &\ge \int_{\T^d}W^{k-1}dx + \frac{U}{2}\int_{\T^d}\big((n^k)^2-(n^{k-1})^2\big)dx 
	\nonumber \\
  &\ge \int_{\T^d}\bigg(W^{k-1}-\frac{U}{2}(n^{k-1})^2\bigg)dx 
	+ \frac{U}{2} \bigg(\int_{\T^d}n^{k-1}dx\bigg)^2. \label{energy.est}
\end{align}
This shows the energy inequality in \eqref{ex.mono}.
Finally, assumption \eqref{ex.assump} gives $\int_{\T^d}W^k dx>0$ and 
consequently $W^k\not\equiv 0$. 

\subsection{An estimate for the variance}	

We claim that the total variance
$$
  V^k := \int_{\T^d}\bigg(W^k-\int_{\T^d}W^k dz\bigg)^2dx
	+ U\int_{\T^d}W^{k-1}dx\int_{\T^d}\bigg(n^k-\int_{\T^d}
	n^kdz\bigg)^2 dx
$$
is nonincreasing in $k$. For the proof, 
we observe that, taking the test function $\phi_1=1$ in the weak formulation
of \eqref{ex.epsw} and performing the limit $\eps\to 0$,
\begin{equation}\label{ex.monow}
  \int_{\T^d}W^k dx \le \int_{\T^d}W^{k-1}dx.
\end{equation}
Thus, by the energy estimate \eqref{energy.est}, 
\begin{align}
  \frac{U}{2}\int_{\T^d}(n^k)^2dx 
  &\le \frac{U}{2}\int_{\T^d}(n^{k-1})^2 dx
	- \int_{\T^d}(W^{k-1}-W^k)dx \nonumber \\
	&\le \frac{U}{2}\int_{\T^d}(n^{k-1})^2 dx. \label{ex.monon}
\end{align}
We employ \eqref{energy.est} again to find that
\begin{align*}
  \bigg(\int_{\T^d} & W^{k-1}dx\bigg)^2 - \bigg(\int_{\T^d}W^k dx\bigg)^2 \\
	&= \bigg(\int_{\T^d} W^{k-1}dx + \int_{\T^d} W^k dx\bigg)
	\bigg(\int_{\T^d} W^{k-1}dx - \int_{\T^d} W^k dx\bigg) \\
	&\le \bigg(\int_{\T^d} W^{k-1}dx + \int_{\T^d} W^k dx\bigg)
	\frac{U}{2}\bigg(\int_{\T^d}(n^{k-1})^2 dx - \int_{\T^d}(n^k)^2 dx\bigg).
\end{align*}
In view of \eqref{ex.monon}, the second bracket on the right-hand side is
nonnegative, such that \eqref{ex.monow} leads to
$$
  \bigg(\int_{\T^d} W^{k-1}dx\bigg)^2 - \bigg(\int_{\T^d}W^k dx\bigg)^2 
	\le U\int_{\T^d}W^{k-1}dx
	\bigg(\int_{\T^d}(n^{k-1})^2 dx - \int_{\T^d}(n^k)^2 dx\bigg).
$$
We take the test function $\phi_1=2\dt$ in \eqref{1.wtau}:
\begin{align*}
  0 &\ge 2\int_{\T^d}(W^k - W^{k-1})W^k dx 
	+ \dt\frac{2d-1}{d}\int_{\T^d}\frac{|\na W^k|^2}{g(n^k)}dx \\
	&\ge \int_{\T^d}(W^k)^2 dx - \int_{\T^d}(W^{k-1})^2 dx
	+ \dt\frac{2d-1}{d}\int_{\T^d}\frac{|\na W^k|^2}{g(n^k)}dx.
\end{align*}
Combining the previous two inequalities, we arrive at
\begin{align*}
  \int_{\T^d} (W^k)^2 dx &- \bigg(\int_{\T^d}W^{k} dx\bigg)^2
	+ \dt\frac{2d-1}{d}\int_{\T^d}\frac{|\na W^k|^2}{g(n^k)}dx \\
	&\le \int_{\T^d}(W^{k-1})^2 dx - \bigg(\int_{\T^d}W^{k-1}dx\bigg)^2 \\
	&\phantom{xx}{}+ U\int_{\T^d}W^{k-1}dx
  \bigg(\int_{\T^d}(n^{k-1})^2 dx - \int_{\T^d}(n^k)^2 dx\bigg).
\end{align*}
Since the measure of $\T^d$ is one, we have
$$
  \int_{\T^d} (W^k)^2 dx - \bigg(\int_{\T^d}W^{k} dx\bigg)^2
	= \int_{\T^d}\bigg((W^k)^2 - \bigg(\int_{\T^d}W^k dz\bigg)\bigg)^2 dx.
$$
Thus, taking into account mass conservation 
$\int_{\T^d}n^k dx=\int_{\T^d}n^{k-1}dx$ and $\int_{\T^d}W^{k-1}dx
\le\int_{\T^d}W^{k-2}dx$ (see \eqref{ex.monow}),
\begin{align*}
  \int_{\T^d} \bigg((W^k)^2 &- \bigg(\int_{\T^d}W^k dz\bigg)\bigg)^2 dx
	+ \dt\frac{2d-1}{d}\int_{\T^d}\frac{|\na W^k|^2}{g(n^k)}dx \\
	&\le \int_{\T^d}\bigg((W^{k-1})^2 - \bigg(\int_{\T^d}W^{k-1} dz\bigg)\bigg)^2 dx \\
	&\phantom{xx}{}+ U\int_{\T^d}W^{k-2}dx
	\int_{\T^d}\bigg((n^{k-1})^2 - \bigg(\int_{\T^d}n^{k-1} dz\bigg)\bigg)^2 dx \\
	&\phantom{xx}{}- U\int_{\T^d}W^{k-1}dx
	\int_{\T^d}\bigg((n^k)^2 - \bigg(\int_{\T^d}n^k dz\bigg)\bigg)^2 dx,
\end{align*}
and the claim follows after using the
lower-semicontinuity of the $L^2$-norm.

\subsection{Proof of Theorem \ref{thm.ex2}}

Let $d=1$. The second inequality in \eqref{ex.mono} implies that
$$
	\int_{\T^1}|\pa_x W^k|^2dx\leq \frac{G}{\dt}V^{k-1},
$$
where $G=\max_{\delta\le s\le\{n^{k-1}\|_{L^\infty(\T)}\}}g(s)
\ge\|g(n^k)\|_{L^\infty(\T)}$.
By the mean-value theorem, there exists $x_0\in\T$ such that
$$
  W^k(x) = W^k(x_0) + \int^x_{x_0}\pa_x W^k(z)dz 
	\ge \int_{\T}W^k(z)dz - \int_{\T}|\pa_x W^k|dz.
$$
Then, using Jensen's inequality and the energy estimate in \eqref{ex.mono},
\begin{align*}
  W^k(x) &\ge \int_\T W^kdx - \int_\T|\pa_x W^k|^2 dx \\
	&\ge \int_\T W^{k-1}dx - \frac{U}{2}\int_\T\bigg(n^{k-1}-\int_\T n^{k-1}dz\bigg)^2dx
	- \frac{G}{\dt}V^{k-1}.
\end{align*}
By definition \eqref{var} of $V^k$, the right-hand side is positive if 
\begin{align*}
  \overline{W^{k-1}}
	> \frac{G}{\dt}\int_\T(W^{k-1}-\overline{W^{k-1}})^2 dx
	+ U\bigg(\frac{G}{\dt}\overline{W^{k-2}} + \frac12\bigg)
	\int_\T(n^{k-1}-\overline{n^{k-1}})^2 dx,
\end{align*}
which is our assumption. Since $W^{k}\in H^1(\T)\hookrightarrow C^0([0,1])$,
we conclude that $W^k>0$ in $[0,1]$. Then we can use $\phi_1=\phi/W^k$ as
a test function in \eqref{1.wtau} and obtain the standard weak formulation
of \eqref{1.wtau} for test functions $\phi\in H^1(\T)$. Furthermore,
for $\phi\in H^1(\T)$,
\begin{align*}
  \int_\T n^k\pa_x\phi dx
	&= \int_\T n^k\pa_x\bigg(W^k\pa_x\bigg(\frac{\phi}{W^k}\bigg)
	+ \pa_x W^k\frac{\phi}{W^k}\bigg)dx \\
	&= -\int_\T\pa_x(n^kW^k)\frac{\phi}{W^k}dx + \int_\T n^k\pa_xW^k\frac{\phi}{W^k}dx,
\end{align*}
showing that $n^k\in H^1(\T)$ and finishing the proof.


\section{Numerical simulations}\label{sec.num}

We solve the one-dimensional equations \eqref{1.n} and \eqref{1.wtot} 
on the torus in conservative form, i.e.\ for
the variables $n$ and $W_{\rm tot}$. The equations are discretized
by the implicit Euler method and solved in a semi-implicit way:
\begin{align}
  \frac{1}{\dt}(n^k-n^{k-1}) 
  &= \pa_x\bigg(\frac{W^{k-1}\pa_x n^k}{n^{k-1}(1-\eta n^{k-1})}
	\bigg), \label{num.n} \\
	\frac{1}{\dt}(W^k_{\rm tot}-W^{k-1}_{\rm tot}) &= \pa_x\bigg(
	\frac{\pa_x W^k}{2n^k(1-\eta n^k)} + \frac{UW^k}{1-\eta n^k}\pa_x n^k\bigg),
	\label{num.w}
\end{align}
where $W^k_{\rm tot}=W^k-(U/2)(n^k)^2$, $x\in\T=(0,1)$.
The spatial derivatives are discretized by centered finite differences
with constant space step $\triangle x>0$.
For given $(W^{k-1},n^{k-1})$, the first equation \eqref{num.n} is solved for $n^k$. 
This solution is employed in the second equation \eqref{num.w}
which is solved for $W^k_{\rm tot}$. Finally, we define
$W^k=W^k_{\rm tot}+(U/2)(n^k)^2$. We choose the parameters $U=10$, $\eta=1$,
$\dt=10^{-5}$, and $\triangle x=10^{-2}$. The initial energy $W^0$ is constant
and the initial density equals
$$
  n^0(x) = \left\{\begin{array}{ll}
	3/4 &\quad\mbox{for }1/4\le x\le 3/4, \\
	1/4 &\quad\mbox{else},
	\end{array}\right. \quad x\in[0,1].
$$

\begin{figure}[ht]
\centering
\includegraphics[width=140mm]{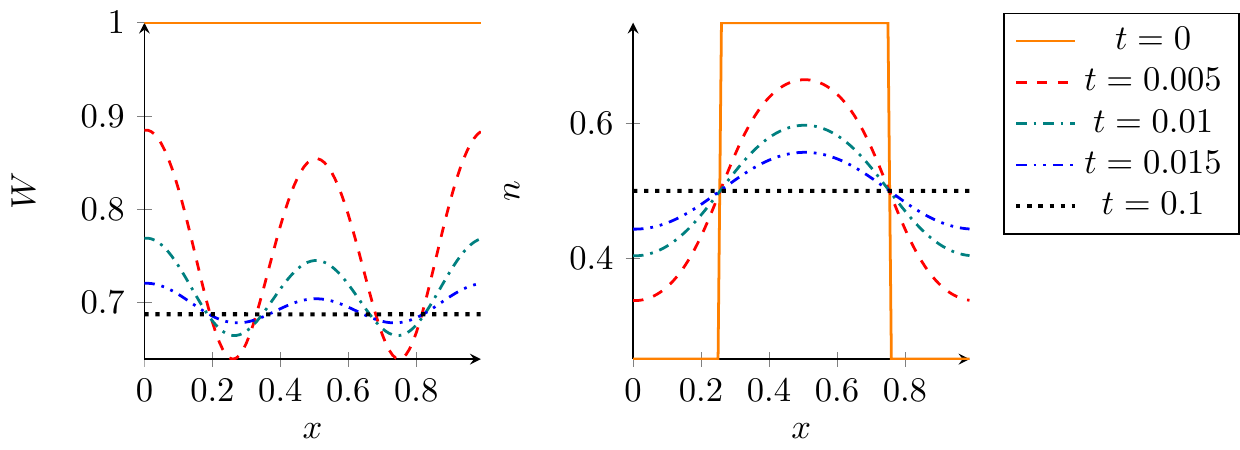}
\caption{Evolution of $(n,W)$ with initial energy $W^0=1$.}
\label{fig.evol1}
\end{figure}

\begin{figure}[ht]
	\centering
	\includegraphics[width=140mm]{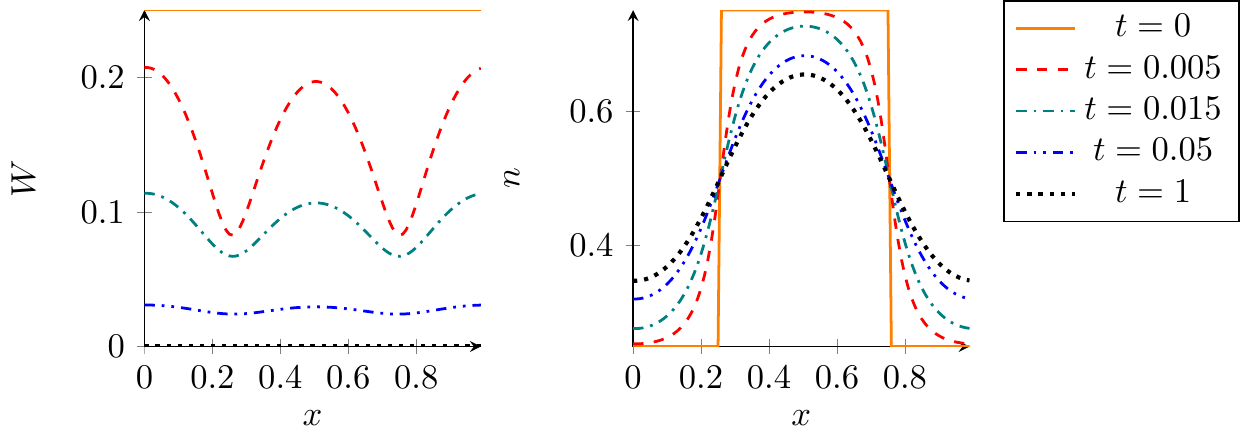}
	\caption{Evolution of $(n,W)$ with initial energy $W^0=1/4$.}
	\label{fig.evol2}
\end{figure}

The time evolution of the particle density and energy is shown in
Figure \ref{fig.evol1} for initial energy $W^0=1$.
The variables converge to the constant steady state 
$(n^\infty,W^\infty)$ as $t\to\infty$, which is almost reached after time $t=0.1$.
Since equations \eqref{num.n}-\eqref{num.w} are conservative,
the total particle number $\int_0^1 n(x,t)dx$ and the total energy
$\int_0^1 W_{\rm tot}(x,t)dx$ are constant in time. Consequently,
the values for the steady state can be computed explicitly. We obtain
for $W^0=1$,
\begin{align*}
  n^\infty &= \int_0^1 n^0(x)dx = \frac12, \\
	W^\infty &= W^\infty_{\rm tot} + \frac{U}{2}(n^\infty)^2
	= \int_0^1\bigg(W^0(x)-\frac{U}{2}n^0(x)^2\bigg)dx + \frac{U}{2}(n^\infty)^2
	= \frac{11}{16}.
\end{align*}
The energy stays positive for all times, so the
high-temperature equations are strictly parabolic, and the convergence
to the (constant) steady state is quite natural. 

The situation is different
in Figure \ref{fig.evol2}, where the particle density converges
to a {\em nonconstant} steady state $n^\infty$ (we have chosen $W^0=1/4$). 
This can be understood as follows. By contradiction, 
let both the particle density and energy be converging
to a constant steady state. Then $n^\infty=1/2$ (see the above calculation)
and
$$
  W^\infty = \int_0^1\bigg(W^0(x)-\frac{U}{2}n^0(x)^2\bigg)dx + \frac{U}{2}(n^\infty)^2
	= -\frac{1}{16}.
$$
However, this contradicts the fact that the energy $W$ is nonnegative
which follows from the maximum principle. Therefore, it is plausible that either $n$
or $W$ cannot converge to a constant. If $n^\infty$ is not constant,
$W^\infty\na n^\infty$ is constant only if $W^\infty=0$. Thus, it is reasonable that
the energy converges to zero, while $n^\infty$ is not constant.
One may say that there is not sufficient initial ``reverted'' energy to
level the particle density.

\begin{figure}[ht]
\centering
\includegraphics[width=140mm]{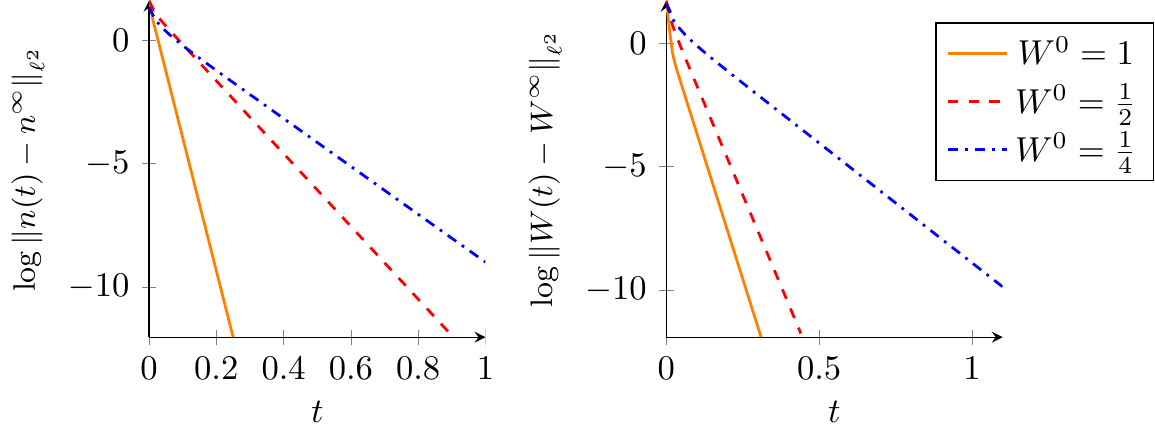}
\caption{Decay rates for various initial energies.}
\label{fig.decay}
\end{figure}

Another difference between Figure \ref{fig.evol1} and Figure \ref{fig.evol2} 
is the time scale. For larger initial energies, the convergence to
equilibrium is faster. In fact, Figure \ref{fig.decay} shows that the
decay of the $\ell^2$ norm of $n(t)-n^\infty$ and $W(t)-W^\infty$
is exponential. Here, we have chosen the initial particle density
$n^0(x)=\frac14$ for $0\le x<\frac12$, $n^0(x)=\frac34$ for $\frac12\le x<1$,
and the initial energy $W^0\in\{\frac14,\frac12,1\}$. 
For $W^0\in\{1,\frac12\}$, we have $n^\infty=\frac12$
and $W^\infty=\max\{0,W^0-U/32\}$.
For $W^0=\frac14$, it holds that $W^\infty=0$ and we have set $n^\infty(x)=n(x,2)$.

Finally, we compute the numerical convergence rates for different
space and time step sizes $\triangle x$ and $\dt$, respectively. Since there is
no explicit solution available, we choose as reference solution the
solution to \eqref{num.n}-\eqref{num.w} with $\triangle x=1/1680$ (for the computation
of the spatial $\ell^2_x$ error) and $\dt=1/5040$ (for the computation of the
$\ell^2_t\ell^2_x$ error). Figure \ref{fig.conv1} shows that the temporal
error is linear in $\dt$, and Figure \ref{fig.conv2}
indicates that the spatial error is quadratic in $\triangle x$.
These values are expected in view of our finite-difference discretization and they
confirm the validity of the numerical scheme.

\begin{figure}[ht]
\centering
\includegraphics[width=110mm]{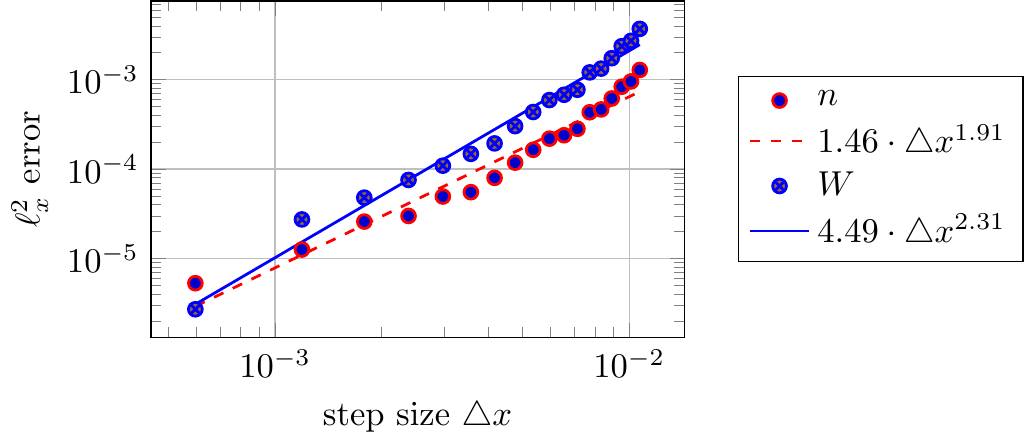}
\caption{Numerical convergence in time.}
\label{fig.conv1}
\end{figure}

\begin{figure}[ht]
\centering
\includegraphics[width=110mm]{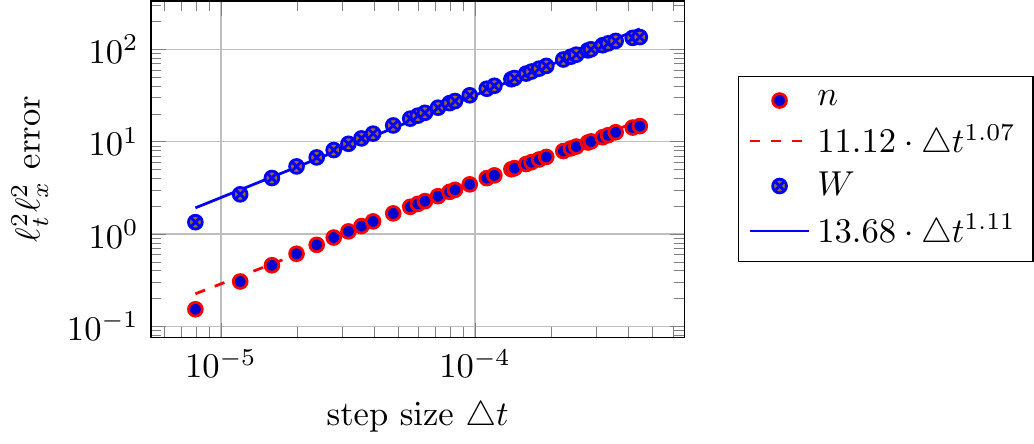}
\caption{Numerical convergence in space.}
\label{fig.conv2}
\end{figure}


\begin{appendix}
\section{Calculation of some integrals}

We recall that $\eps(p)=-2\eps_0\sum_{k=1}^d\cos(2\pi p_i)$. Then
$u_i(p)=(\pa\eps/\pa p_i)(p)=4\pi\eps_0\sin(2\pi p_i)$, and we calculate
\begin{align}
  \int_{\T^2}\eps^2 dp &= 4\eps_0^2\int_0^1\sum_{k=1}^d\cos^2(2\pi p_k)dp_k
	= 2d\eps_0^2, \label{2.eps2} \\
  \int_{\T^d} u_iu_j dp &= (4\pi\eps_0)^2\delta_{ij}\int_0^1\sin^2(2\pi p_i) dp_i
	= \frac12(4\pi\eps_0)^2\delta_{ij}, \label{2.uu} \\
  \int_{\T^d} \eps u_iu_j dp &= -2\eps_0(4\pi\eps_0)^2\sum_{k=1}^d\int_{\T^d}
	\cos(2\pi p_k)\sin(2\pi p_i)\sin(2\pi p_j)dp \label{2.epsuu} \\
	&= -2\eps_0(4\pi\eps_0)^2\delta_{ij}\int_0^1\cos(2\pi p_i)\sin^2(2\pi p_i)dp_i = 0, 
	\nonumber \\
	\int_{\T^d} \eps^2 u_iu_j dp &= \frac13\int_{\T^d}\frac{\pa}{\pa p_i}(\eps^3)
	\frac{\pa\eps}{\pa p_j}dp 
	= -\frac13\int_{\T^d}\eps^3\frac{\pa^2\eps}{\pa p_i\pa p_j}dp = 0
	\quad\mbox{if }i\neq j, \nonumber
\end{align}
since $\pa^2\eps/\pa p_i\pa p_j=0$ for $i\neq j$. 
We compute the integral $\int_{\T^d}\eps^2 u_i^2dp$.
First, let $d=1$. Then
$$
  \int_{\T} \eps^2 u_1^2 dp_1 
	= 4\eps_0^2(4\pi\eps_0)^2\int_0^1\cos^2(2\pi p_1)\sin^2(2\pi p_1)dp_1 
	= \frac{\eps_0^2}{2}(4\pi\eps_0)^2 = 8\pi^2\eps_0^4.
$$
Furthermore, for $d>1$, 
\begin{align*}
  \int_{\T^d} & \eps^2 u_i^2 dp 
	= 4\eps_0^2(4\pi\eps_0)^2\int_{\T^d}\bigg(\sum_{k=1,\,k\neq i}^d
	\cos(2\pi p_k) + \cos(2\pi p_i)\bigg)^2\sin^2(2\pi p_i)dp \\
	&= 4\eps_0^2(4\pi\eps_0)^2\left(\int_{\T^{d}}\bigg(\sum_{k=1,\,k\neq i}^d
	\cos(2\pi p_k)\bigg)^2\sin^2(2\pi p_i)dp
	+ \int_\T \cos^2(2\pi p_i)\sin^2(2\pi p_i)dp_i\right) \\
	&= 4\eps_0^2(4\pi\eps_0)^2 \int_{T^{d}}
	\sum_{k=1,\,k\neq i}^d\cos^2(2\pi p_k)\sin^2(2\pi p_i)dp
	+ 8\pi^2\eps_0^4 \\
	&= 4\eps_0^2(4\pi\eps_0)^2\sum_{k=1,\,k\neq i}^d\int_0^1\cos^2(2\pi p_k)dp_k
	\int_0^1\sin^2(2\pi p_i)dp_i + 8\pi^2\eps_0^4 \\
	&= (d-1)\eps_0^2(4\pi\eps_0)^2 + 8\pi^2\eps_0^4 = 8(2d-1)\pi^2\eps_0^4.
\end{align*}
We conclude that
\begin{equation}\label{2.eps2uu}
  \int_{\T^d}\eps^2 u_iu_j dp = 8(2d-1)\pi^2\eps_0^4\delta_{ij}.
\end{equation}

\end{appendix}


\end{document}